%% file: wp_and_scc_Main.tex
\documentclass[a4paper,11pt]{article} 
\usepackage{fancyhdr}
\usepackage{graphicx}
\usepackage{microtype} 
\usepackage[shortlabels]{enumitem}

\usepackage[backend=bibtex,citestyle=alphabetic,bibstyle=alphabetic,backref=true,url=false]{biblatex}
\usepackage{amsfonts,amssymb,amstext,amsmath} 
\usepackage{cool} 
\usepackage[amsmath,thmmarks,hyperref]{ntheorem}
\usepackage{tensor}
\usepackage{tensind} 
\tensordelimiter{?}

\usepackage[bookmarksnumbered,hypertexnames=false]{hyperref} 
\hypersetup{
    plainpages=false,       
    unicode=false,          
    pdftoolbar=true,        
    pdfmenubar=true,        
    pdffitwindow=false,     
    pdfstartview={FitH},    
    pdftitle={Warped products and Spaces of Constant Curvature},    
    pdfauthor={Krishan Rajaratnam},     
    pdfnewwindow=true,      
    linktoc=page,           
    colorlinks=true,        
    linkcolor=blue,         
    citecolor=black,        
    filecolor=magenta,      
    urlcolor=cyan           
}
\usepackage{bookmark}

\usepackage[capitalize]{cleveref} 
\input{commands2.0.tex}
\input{ntheorem.tex}
\renewcommand{\d}{{\textrm d}}

\usepackage{autonum} 

\usepackage[sanitize={symbol=false},style=index,sort=def,toc,nomain]{glossaries} 
\newglossary[nlg]{notation}{not}{ntn}{List of Notations} 
\loadglsentries{HJ-theory-glossaries}
\makeglossary

\begin{document}
\pagenumbering{roman}
\title{Warped products and Spaces of Constant Curvature}
\author{Krishan Rajaratnam\footnote{Department of Applied Mathematics, University of Waterloo, Canada; e-mail: k2rajara@uwaterloo.ca}}
\date{\today} 
\maketitle
\begin{abstract}\centering
	We will obtain the warped product decompositions of spaces of constant curvature (with arbitrary signature) in their natural models as subsets of pseudo-Euclidean space. This generalizes the corresponding result by \citeauthor{Nolker1996} in \cite{Nolker1996} to arbitrary signatures, and has a similar level of detail. Although our derivation is complete in some sense, none is proven. Motivated by applications, we will give more information for the spaces with Euclidean and Lorentzian signatures. This is an expository article which is intended to be used as a reference. So we also give a review of the theory of circles and spheres in pseudo-Riemannian manifolds.
\end{abstract}
\tableofcontents
\cleardoublepage
\phantomsection		

\printglossaries
\cleardoublepage
\phantomsection


\addcontentsline{toc}{section}{List of Results}
\textbf{List of Results}
\theoremlisttype{allname} 
\listtheorems{theorem,proposition,corollary}
\cleardoublepage
\phantomsection

\pagenumbering{arabic}
\newpage
\include{wp_and_scc_Body}
\phantomsection
\addcontentsline{toc}{section}{References}
\printbibliography
\end{document}

%% file: commands2.0.tex
\usepackage{amsmath,amsfonts,amssymb}

\newcommand{\spa}[1]{\operatorname{span} \{#1\}}
\newcommand{\End}{\operatorname{End}}
\newcommand{\Ima}{\operatorname{Im}}

\newcommand{\scalprod}[2]{\left\langle #1,#2 \right\rangle}
\newcommand{\tr}[1]{\operatorname{tr}(#1)}

\def\sp(#1,#2){\left\langle #1,#2 \right\rangle}
\def\bp#1{\sp(#1)}
\newcommand{\norm}[1]{\left\| #1 \right\|} 
\newcommand{\ind}{\operatorname{ind}} 

\DeclareFontFamily{U}{matha}{\hyphenchar\font45}
\DeclareFontShape{U}{matha}{m}{n}{
      <5> <6> <7> <8> <9> <10> gen * matha
      <10.95> matha10 <12> <14.4> <17.28> <20.74> <24.88> matha12
      }{}
\DeclareSymbolFont{matha}{U}{matha}{m}{n}
\DeclareFontFamily{U}{mathx}{\hyphenchar\font45}
\DeclareFontShape{U}{mathx}{m}{n}{
      <5> <6> <7> <8> <9> <10>
      <10.95> <12> <14.4> <17.28> <20.74> <24.88>
      mathx10
      }{}
\DeclareSymbolFont{mathx}{U}{mathx}{m}{n}

\DeclareMathSymbol{\obot}         {2}{matha}{"6B}
\DeclareMathSymbol{\bigobot}       {1}{mathx}{"CB}

\renewcommand{\P}{\mathbb{P}}
\newcommand{\R}{\mathbb{R}}

\renewcommand{\E}{\mathbb{E}} 
\newcommand{\Si}{\mathbb{S}}

\newcommand{\eunn}{\mathbb{E}^{n}_{\nu}}

\newcommand{\punn}{\mathbb{P}^{n}_{\nu}}
\newcommand{\sgn}{\operatorname{sgn}}

\def\ve{\mathfrak{X}} 

\newcommand{\Ei}{{\mathcal{E}}}
\newcommand{\F}{{\mathcal{F}}}


%% file: ntheorem.tex
\usepackage{amsfonts,amssymb,amsmath}

\theoremnumbering{arabic}
\theoremstyle{break} 
\usepackage{latexsym}
\theoremsymbol{\ensuremath{_\Box}}
\theorembodyfont{\itshape}
\theoremheaderfont{\normalfont\bfseries}
\theoremseparator{}
\newtheorem{theorem}{Theorem}[section]
\newtheorem{proposition}[theorem]{Proposition}
\newtheorem{corollary}[theorem]{Corollary}
\newtheorem{lemma}[theorem]{Lemma}
\theoremsymbol{\ensuremath{\diamondsuit}} 

\theoremsymbol{\ensuremath{_\Box}}

\theorembodyfont{\upshape}
\newtheorem{example}[theorem]{Example}

\newtheorem{remark}[theorem]{Remark}
\newtheorem{definition}[theorem]{Definition}

\newtheorem{claim}{Claim}[theorem]

\theoremstyle{nonumberplain}
\theoremheaderfont{\scshape}
\theorembodyfont{\normalfont}
\theoremsymbol{\ensuremath{_\blacksquare}}
\usepackage{amssymb}
\newtheorem{proof}{Proof}

\qedsymbol{\ensuremath{_\blacksquare}}

\numberwithin{equation}{section} 

\newcounter{partCounter}
\newenvironment{parts}{
	\begin{list}{\bfseries{}Case \arabic{partCounter}~}{\usecounter{partCounter}}
}{
	\end{list}
}

%% file: HJ-theory-glossaries.tex

\newglossaryentry{tenRCurv}{name=Riemann curvature tensor, description={
$	R(X,Y)Z = \nabla_{X}\nabla_{Y} - \nabla_{Y}\nabla_{X} - \nabla_{[X,Y]}$}}

\newglossaryentry{tenHess}{name=Hessian of a function f, description={
\begin{equation*}
	H^{f}(X,Y) = XY f - (\nabla_{X} Y) f
\end{equation*}
}}

\newglossaryentry{meanCVf}{name=Mean curvature vector field, description={
\begin{equation*}
	H = \frac{1}{n} \sum\limits_{i=1}^{n} \epsilon_{i} \mathbf{II}
\end{equation*}
}}

\newglossaryentry{ot}{name={orthogonal tensor},
description={A valence 2 symmetric contravariant tensor whose uniquely determined endomorphism is point-wise diagonalizable.}}

\newglossaryentry{nTorLess}{name={torsionless},
description={A $\binom{1}{1}$-tensor is called torsionless if it's Nijenhuis torsion vanishes.}}


\newglossaryentry{eunn}{type=notation,name={\ensuremath{\eunn}}, description={pseudo-Euclidean space, an $n$-dimensional vector space equipped with a metric with signature $\nu$}} 

\newglossaryentry{eunnk}{type=notation,name={\ensuremath{\eunn(\kappa)}}, description={ A hyperquadric of pseudo-Euclidean space. More precisely the central hyperquadric of \gls{eunn} with curvature $\kappa$.}}

\newglossaryentry{obot}{type=notation,
name={\ensuremath{\obot}},
description={The orthogonal direct sum.
}}

\newglossaryentry{sgn}{type=notation,
name={\ensuremath{\sgn}},
description={Given a real number $a$, $\sgn a$ is the sign of $a$ if $a \neq 0$ and $0$ if $a = 0$.
}}



\newglossaryentry{fm}{type=notation,
name={\ensuremath{\F(M)}},
description={The set of functions defined on the manifold $M$.
}}

\newglossaryentry{vem}{type=notation,
name={\ensuremath{\ve(M)}},
description={The set of vector fields defined on the manifold $M$.
}}

\newglossaryentry{secd}{type=notation,
name={\ensuremath{\Gamma(E)}},
description={The set of vector fields tangent to the distribution $E$.
}}

\newglossaryentry{spm}{type=notation,
name={\ensuremath{S^{p}(M)}},
description={The set of symmetric contravariant tensors of valence $p$ defined on the manifold $M$.
}}

\newglossaryentry{con}{type=notation,
name={\ensuremath{\operatorname{C}^{p}(M)}},
description={The vector space of concircular contravariant tensors of valence $p$ defined on the manifold $M$.
}}

\newglossaryentry{cono}{type=notation,
name={\ensuremath{\operatorname{C}^p_0(M)}},
description={The vector space of covariantly constant contravariant tensors of valence $p$ defined on the manifold $M$.
}}

\newglossaryentry{odot}{type=notation,
name={\ensuremath{\odot}},
description={The symmetric product of two tensors, i.e. if $u,v$ are tensors then $u \odot v$ is the symmetrization of $u \otimes v$.},
sort=od}

\newglossaryentry{end}{type=notation, name={\ensuremath{\End}},
description={We refer to a linear map from a vector space (vector bundle) into itself as an \emph{endomorphism}.}}


\newacronym{scc}{SCC}{space of constant curvature}

\newacronym{kv}{KV}{Killing vector}
\newacronym{ckv}{CKV}{conformal Killing vector}
\newacronym{kt}{KT}{Killing tensor}
\newacronym{chkt}{ChKT}{characteristic Killing tensor}
\newacronym{ckt}{CKT}{conformal Killing tensor}
\newacronym{cv}{CV}{concircular vector}
\newacronym{ct}{CT}{concircular tensor also called a C-tensor}
\newacronym{oct}{OCT}{orthogonal concircular tensor}
\newacronym{ict}{ICT}{irreducible concircular tensor}

\newacronym{kss}{KS-space}{Killing-Stackel space}
\newacronym{kem}{KEM}{Kalnins-Eisenhart-Miller}
\newacronym{bekm}{BEKM}{Benenti-Eisenhart-Kalnins-Miller}
\newacronym{kbd}{KBD}{Killing Bertrand-Darboux}
\newacronym{kbdt}{KBDT}{Killing Bertrand-Darboux tensor}

\newacronym{wpnet}{WP-net}{warped product net}
\newacronym{tpnet}{TP-net}{twisted product net}

\newacronym{hj}{HJ}{Hamilton-Jacobi}

%% file: wp_and_scc_Body.tex
\section{Introduction}

Warped products are ubiquitous in applications of pseudo-Riemannian geometry. Most of the separable coordinate systems in spaces of constant curvature are built up using them \cite{Kalnins1986}, and some exact solutions in general relativity are composed of them \cite{Dobarro2005,Zeghib2011}. They can intuitively be thought of as a partial generalization of the spherical coordinate system to arbitrary pseudo-Riemannian manifolds. Indeed, it can be shown that all the spherical coordinate systems (on any space of constant curvature) can be constructed iteratively using warped products, and that they share several properties with these coordinate systems. Similarly the well known Schwarzschild metric in relativity can be constructed using warped products.

Various geometrical objects take canonical forms in warped products. For example, one can calculate general formulas for the Levi-Civita connection and the Riemann curvature tensor in a warped product \cite{Meumertzheim1999}. These product manifolds can be used to construct geometrical objects with special properties. For example, it was shown in \cite{Rajaratnam2014a}, that one can use the warped product decompositions of a given space to construct Killing tensors and hence coordinates which separate the Hamilton-Jacobi equation. Thus it is only natural that we determine the warped products which are isometric to spaces of constant curvature.

We now describe more precisely the problem we solve, after introducing some definitions. A \emph{warped product} is a product manifold $M = \prod_{i=0}^{k} M_{i}$ of pseudo-Riemannian manifolds $(M_{i}, g_{i})$ where $\dim M_{i} > 0$ for $i > 0$ equipped with the metric

\begin{equation}
	g = \pi_{0}^{*} g_{0} + \sum_{i=1}^{k} \rho_{i}^{2} \pi_{i}^{*} g_{i}
\end{equation}

\noindent where $\rho_{i} : M_0 \rightarrow \R^{+}$ are functions and $\pi_{i} : M \rightarrow M_{i}$ are the canonical projection maps \cite{Meumertzheim1999}. The warped product is denoted by $M_0 \times_{\rho_1} M_1 \times \cdots \times_{\rho_k} M_k $. We say a warped product is a \emph{warped product decomposition} of a pseudo-Riemannian manifold $M$ if it is isometric to some non-empty open subset of $M$. In this article we will present an interesting class of warped product decompositions of spaces of constant curvature (with arbitrary signature).

Our solution follows that by \citeauthor{Nolker1996} in \cite{Nolker1996}, which is for the special case of Riemannian spaces of constant curvature. We make use of the observation that for a warped product $M = M_0 \times_{\rho_1} M_1 \times \cdots \times_{\rho_k} M_k $, $M_0$ is a \emph{geodesic submanifold} of $M$ and for each $i > 0$ the manifold $M_i$ is a \emph{spherical submanifold} of $M$\footnote{Often a geodesic submanifold is called totally geodesic and a spherical submanifold is called an extrinsic sphere.} \cite{Meumertzheim1999}. Thus after characterizing all geodesic and spherical submanifolds of spaces of constant curvature, we generalize a formula given in \cite{Nolker1996} to obtain an interesting class of warped product decompositions.

Our primary motivation for this work comes from \cite{Rajaratnam2014a}, where it was shown that warped products can be used to construct coordinates which separate the Hamilton-Jacobi equation. Based on this application, it will become clear (in a following article) that our work is ``complete''. We are mainly interested in exposing these results for reference purposes. Since as of now, it is difficult to find any articles/books which can be used as a reference for our purposes. For a similar reason, we will present a review of the theory of circles and spheres in pseudo-Riemannian manifolds as well.

This article is mostly self-contained, so it can be used as a reference. However, we use some results from the theory of pseudo-Riemannian submanifolds in \cite{chen2011pseudo}, which are only necessary to understand certain proofs. We also assume the reader is familiar with \cite{barrett1983semi}, especially with the basic properties of pseudo-Euclidean vector spaces and (pseudo-)Riemannian submanifold theory. Familiarity with the article \cite{Nolker1996} is useful but not necessary.

The article is organized as follows. After defining basic notations in \cref{sec:notNCon}, we summarize our results in \cref{sec:WPsNSCC}. This summary should be sufficient for applications. The subsequent sections provide proofs and more details. In \cref{sec:pRiemSubFol} we give a brief review of the theory of pseudo-Riemannian submanifolds/foliations. In \cref{sec:circSphSub} we apply these concepts by reviewing the theory of circles and spheres in pseudo-Riemannian manifolds. This section is optional but it gives a geometric interpretation of warped products and is included because there are relatively few reviews of this topic. In \cref{sec:StdSubEunn,sec:SphSubSCC,sec:WP,sec:wpSCC} we review preliminary theory on the spherical submanifolds and warped products in spaces of constant curvature and warped products in general. We give the warped product decompositions of pseudo-Euclidean space in \cref{sec:WPdecompEunn} and of spherical submanifolds of pseudo-Euclidean space in \cref{sec:WPdecompSphEunn}. \Cref{sec:IsoSphEunn} is another optional section which gives the isometry groups of spherical submanifolds of pseudo-Euclidean space, referring to \cite{barrett1983semi} in the appropriate cases.

\section{Notations and Conventions} \label{sec:notNCon}

All differentiable structures are assumed to be smooth (class $C^{\infty}$). Let $M$ be a pseudo-Riemannian manifold of dimension $n$ equipped with covariant metric $g$. Unless specified otherwise, it is assumed that $n \geq 2$. The contravariant metric is usually denoted by $G$ and $\bp{\cdot, \cdot}$ plays the role of the covariant and contravariant metric depending on the arguments. We denote by $\F(M)$ the set of functions from M to $\R$ and $\ve(M)$ denotes the set of vector fields over $M$. If $x \in \ve(M)$ then we denote $x^2 := \bp{x,x}$ and $\norm{x} := \sqrt{\Abs{\bp{x,x}}}$.

Throughout this article we will be working in pseudo-Euclidean space, which is defined as follows. An $n$-dimensional vector space $V$ equipped with metric $g$ of signature\footnote{The signature is equal to the number of negative diagonal entries in a basis which diagonalizes $g$.} $\nu$ is denoted by \gls{eunn} and called \emph{pseudo-Euclidean space}. We obtain Euclidean space $\E^n$ in the special case where $\nu = 0$. Also Minkowski space $M^n$ is obtained by taking $\nu = 1$. A subspace $U \subseteq \eunn$ is called \emph{non-degenerate} if the induced metric is non-degenerate and in this case we denote by $\ind U$ the signature of the induced metric. Also note that since $\eunn$ is a vector space, for any $p \in \eunn$ we identify vectors in $T_p \eunn$ with points in $\eunn$.

Given an open subset $U \subseteq \eunn$ and $\kappa \in \R \setminus \{0\}$, we denote by $U(\kappa)$ the \emph{central hyperquadric} of $\eunn$ contained in $U$, which is defined by:

\begin{equation}
	U(\kappa) = \{p \in U \; | \; \bp{p,p} = \kappa^{-1} \}
\end{equation}

Usually $U = \eunn$ and this is denoted \gls{eunnk}. The notation $U(\kappa)^\circ$ represents a maximal connected component of $U(\kappa)$. It is well known that $\eunn(\kappa)$ is a pseudo-Riemannian manifold of dimension $n-1$ with signature $\nu + \frac{(\gls{sgn} \kappa -1)}{2}$ and constant curvature\footnote{This will be proven later.} $\kappa$ \cite{barrett1983semi}. Since $\eunn(\kappa) \subset \eunn$, for any $p \in \eunn(\kappa)$ we identify vectors in $T_p \eunn(\kappa)$ with points in $\eunn$. Occasionally we use the following conventions: If $\kappa = 0$, we set $\eunn(0) := \eunn$, if $\kappa = \infty$ we set $\eunn(\infty)$ to be the light cone, i.e. the set of non-zero null vectors. We also use the following notations: If $\kappa > 0$ then $S_{\nu}^n(\kappa) := \E^{n+1}_{\nu}(\kappa)^\circ$, if $\kappa < 0$ then $H_{\nu}^n(\kappa) := \E^{n+1}_{\nu+1}(\kappa)^\circ$.

We define the parabolic embedding of $\eunn$ in $\E^{n+2}_{\nu+1}$ with mean curvature vector $-a \in \eunn(\infty)$ by \cite{Tojeiro2007}

\begin{equation}
	\punn := \{p \in \E^{n+2}_{\nu+1}(\infty) : \bp{p, a} = 1 \}
\end{equation}

An explicit isometry with $\eunn$ is obtained by choosing $b \in \punn$, i.e. $b$ is lightlike and $\bp{a, b} = 1$. We let $V := \spa{a,b}^\perp$, note that $V \cong \eunn$, then for $x \in V$:

\begin{equation} \label{eq:psiPunn}
	\psi(x) = b + x  - \frac{1}{2}x^2a \in \punn
\end{equation}

More details on the properties of $\punn$ will be given later on (see \cref{prop:eunnPunn}). Finally, we define the \emph{dilatational vector field} in $\eunn$, $r$, to be the vector field satisfying for any $p \in \eunn$, $r_p = p \in T_p \eunn$.

\section{Warped products in Spaces of Constant Curvature} \label{sec:WPsNSCC} 

In this section we will briefly describe the warped product decompositions of spaces of constant curvature, in a way which is useful for applications. The proofs of many of the assertions will come in the following sections. We will use the notation $\eunn(\kappa)$ (where $\kappa$ can be zero) to represent the general space of constant curvature. First we will need to know the spherical submanifolds of these spaces.

\begin{theorem}[Spherical submanifolds of $\eunn(\kappa)$] \label{thm:spherSubI}
	Let $\overline{p} \in \eunn(\kappa)$ be arbitrary, $V \subset T_{\bar{p}}\eunn(\kappa)$ a non-degenerate subspace with $m := \dim V \geq 1$, $\mu := \ind V$ and $z \in V^{\perp} \cap T_{\bar{p}}\eunn(\kappa)$. Let $a := \kappa \bar{p} -z$, $\tilde{\kappa} := a^2$ and $W := \R a \obot V$. There is exactly one $m$-dimensional connected and geodesically complete spherical submanifold $\tilde{N}$ with $\bar{p} \in \tilde{N}$, $T_{\bar{p}} \tilde{N} = V$ and having mean curvature vector at $\bar{p}$, z. $\tilde{N}$ is an open submanifold of N; N is referred to as the spherical submanifold determined by $(\bar{p},V,a)$, it is geodesic iff $z = 0$ and is given as follows (where $\simeq$ means isometric to):
	
	\begin{enumerate}[(a),style=multiline]
		\item \label{it:sphEunnaI} $a = 0$, in this case $N \simeq \E^m_{\mu}$ \\
		\begin{equation}
			N = \overline{p} + V
		\end{equation}
		\item \label{it:sphEunnbI} $a$ is timelike, then $\mu \leq \nu - 1$ and $N \simeq H^m_{\mu}(\tilde{\kappa})$
		\item \label{it:sphEunncI} $a$ is spacelike, then $N \simeq S^m_{\mu}(\tilde{\kappa})$ \\
		For cases (b) and (c), let $c = \overline{p} -\frac{a}{\tilde{\kappa}}$ be the center of N, then N is given as follows:
		\begin{equation}
			N = c + \{ p \in W \: | \: p^2 = \frac{1}{\tilde{\kappa}} \}
		\end{equation}
		\item \label{it:sphEunndI} $a$ is lightlike, then $\mu \leq \nu - 1$ and $N \simeq \E^m_{\mu}$
		\begin{equation}
			N = \bar{p} + \{ p - \frac{1}{2} p^2 a  \: | \: p \in V \}
		\end{equation}
	\end{enumerate}
\end{theorem}
\begin{remark}
	If $a$ is lightlike, then $N$ is isometric to $\P^m_{\mu}$ with mean curvature vector $-a$. Furthermore, let $b \in V^\perp$ be a lightlike vector satisfying $\bp{a,b} = 1$. Then the orthogonal projector onto $V$, $P$, induces an isometry of $N - \bar{p} + b$ onto $V$.
\end{remark}
\begin{remark}
	One can find more details on when $N$ is connected in the remarks following \cref{thm:spherSub,thm:spherKapSub}.
\end{remark}
\begin{proof}
	See \cref{thm:spherSub,thm:spherKapSub}.
\end{proof}

With the knowledge of these spherical submanifolds, we can now specify how to construct warped products in $\eunn(\kappa)$. This construction depends on the following data: A point $\bar{p} \in \eunn(\kappa)$, a decomposition $T_{\bar{p}} \eunn(\kappa) = \bigobot\limits_{i=0}^k V_i$ into non-trivial (hence non-degenerate) subspaces with $k \geq 1$, and vectors $z_1,\dotsc,z_k \in V_0$ such that the vectors $a_i := \kappa \bar{p} - z_i$ are pair-wise orthogonal and independent. We call the data $(\bar{p};\bigobot\limits_{i=0}^k V_i; a_1,...,a_k)$, \emph{initial data} for a (proper) warped product decomposition of $\eunn(\kappa)$. If $\kappa = 0$, one can more generally let some of the $a_i$ be zero, this results in Cartesian products as done in \cite{Nolker1996}. Since we assume the $a_i$ are non-zero, we sometimes use the additional qualifier ``proper''.

With this initial data, for $i > 0$ let $N_i$ be the sphere in $\eunn(\kappa)$ determined by $(\bar{p}, V_i, a_i)$ and $\rho_i(p_0) = 1 + \bp{a_i, p_0 - \overline{p}}$. Let $N_0$ be the subset of the sphere in $\eunn(\kappa)$ determined by $(\bar{p}, V_0, \kappa \bar{p})$ where each $\rho_i > 0$. Then the data $(\bar{p}; \bigobot\limits_{i=0}^{k} V_{i}; a_{1},...,a_{k})$, induces a warped product decomposition (of $\eunn(\kappa)$) given as follows:

\begin{equation} \label{eq:genPsiEqnI}
	\psi : \begin{cases}
	N_{0} \times _{\rho_{1}} N_{1} \times \cdots  \times _{\rho_{k}} N_{k} & \rightarrow \eunn(\kappa) \\
	(p_{0},...,p_{k}) & \mapsto p_{0} + \sum\limits_{i=1}^{k} \rho_{i}(p_{0})(p_{i} - \overline{p})
	\end{cases}
\end{equation}

\noindent We note that $\psi$ has the property that $\psi(\bar{p},\dotsc,p_i,\bar{p},\dotsc,\bar{p}) = p_i$.  Often the point $\bar{p}$ doesn't enter calculations, hence we will usually omit it. We note that the above formula generalizes one given in \cite{Nolker1996}.

For actual calculations, it will be more convenient to work with canonical forms. The following definition will be particularly useful.

\begin{definition}[Canonical form for Warped products of $\eunn$]
	We say that a proper warped product decomposition of $\eunn$ determined by $(\bar{p}; \bigobot\limits_{i=0}^k V_i; a_1,...,a_k)$ is in canonical form if: $\bar{p} \in V_0$ and $\bp{\bar{p}, a_i} = 1$.
\end{definition}

Any proper warped product decomposition $\psi$ of $\eunn$ can be brought into canonical form, see the discussion preceding \cref{cor:WPconForm} for details.

We will now give more information on standard warped product decompositions of $\eunn$ in canonical form. Suppose the initial data $(\bar{p}; V_0 \obot V_1 ; a)$ is in canonical form, and let $\psi$ be the associated warped product decomposition given by \cref{eq:genPsiEqnI}. Denote $\kappa := a^2$ and $\epsilon := \sgn \kappa$. We have two types of warped products:

\begin{description}
	\item[non-null warped decomposition] If $\kappa \neq 0$, let $W_0 := V_0 \cap a^\perp$ and $W_1 := W_0^\perp$.
	
	\item[null warped decomposition] If $\kappa = 0$, then $a$ is lightlike, so fix another lightlike vector $b \in V_0$ such that $\bp{a,b} = 1$, let $W_0 := V_0 \cap \spa{a,b}^\perp$ and $W_1 := V_1$.
\end{description}

For $i = 0,1$, let $P_i : \eunn \rightarrow W_{i}$ be the orthogonal projection. Then the following holds:

\begin{theorem}[Standard Warped Products in $\eunn$ \cite{Nolker1996}] \label{thm:WPDecompsI}
	Let $\psi$ be the warped product decomposition of $\eunn$ determined by the initial data $(\bar{p}; V_0 \obot V_1 ; a)$ given above. Then $N_0$ has the following form:
	
	\begin{equation}
		N_{0} = \{p \in V_{0} | \bp{a, p} > 0  \}
	\end{equation}
	
	and
	
	\begin{equation}
		\rho : \begin{cases}
			N_{0} & \rightarrow \R_{+} \\
			p_{0} & \mapsto \bp{a, p_{0}}
		\end{cases}
	\end{equation}
	
	 The map $\psi$ is an isometry onto the following set:
		
	\begin{equation}
		\Ima(\psi) = \begin{cases}
		 \{ p \in \eunn \: | \: \sgn (P_{1}p)^{2} = \epsilon \} & \text{non-null case} \\
		 \{ p \in \eunn \: | \: \bp{a,p} > 0 \} & \text{null case}
		\end{cases}
	\end{equation}
	
	Furthermore, the following equation holds:
	
	\begin{equation} \label{eq:wpdecpSpI}
		\psi(p_{0},p_{1})^{2} = p_{0}^{2}
	\end{equation}
\end{theorem}
\begin{proof}
	See \cref{cor:WPconForm}.
\end{proof}

In fact, for $(p_{0},p_{1}) \in N_{0} \times N_{1}$, $\psi$ has one of the following forms, first if $\psi$ is non-null:

\begin{equation} \label{eq:WPDecNnI}
		\psi(p_{0},p_{1}) = P_{0}p_{0} + \bp{a, p_{0}} (p_{1}-c)
\end{equation}

where $c = \bar{p} - \frac{a}{a^2}$, and if $\psi$ is null:

\begin{equation} \label{eq:WPDecNI}
		\psi(p_{0},p_{1}) = P_{0}p_{0} + (\bp{b, p_{0}} - \frac{1}{2}\bp{a, p_{0}}(P_{1}p_1)^{2})a + \bp{a,p_{0}}b + \bp{a, p_{0}} P_{1}p_1
\end{equation}

The above forms are obtained from the equation for $\psi$ from the above theorem by expanding $p_0$ in an appropriate basis. We note that the warped products with multiple spherical factors can be obtained using the standard ones described above. Indeed, suppose $\phi_1 : N_0' \times_{\rho_1} N_1 \rightarrow \eunn$ is the warped product decomposition determined by $(\bar{p}; V_0 \obot V_1 ; a_1)$ as above. Since $V_0$ is pseudo-Euclidean, consider a warped product decomposition, $\phi_2 : \tilde{N}_0 \times_{\rho_2} N_2 \rightarrow V_0$, determined by $(\bar{p}; \tilde{V}_0 \obot \tilde{V}_1 ; a_2)$ with $V_0 \cap W_0^\perp \subset \tilde{W}_0$ (hence $a_1 \in \tilde{W}_0$). Note that $\tilde{W}_0$ is the subspace $W_0$ from the above construction for $\phi_2$. Let $N_0 := N_0' \cap \tilde{N}_0$, then one can check that the map $\psi$ defined by:

\begin{equation}
	\psi : \begin{cases}
		N_{0} \times_{\rho_1} N_1 \times_{\rho_2} N_2 & \rightarrow \eunn \\
		(p_0 , p_1 , p_2) & \mapsto \phi_1(\phi_2(p_0, p_2), p_1)
	\end{cases}
\end{equation}

\noindent is a warped product decomposition of $\eunn$ satisfying \cref{eq:genPsiEqnI}. We illustrate this construction with an example.

\begin{example}[Constructing multiply warped products]
	Suppose $\phi_1$ and $\phi_2$ are given as follows:
	
	\begin{align}
		\phi_1(p_{0}',p_{1}) & = P_{0}'p_{0}' + \bp{a_1, p_{0}'} (p_{1}-c_1) \\
		\phi_2(\tilde{p}_{0},p_{2}) & = \tilde{P}_{0}\tilde{p}_{0} + \bp{a_2, \tilde{p}_{0}} (p_{2}-c_2)
	\end{align}
	
	Now observe that $\rho_1(\phi_2(\tilde{p}_{0},p_{2})) = \rho_1(\tilde{p}_{0})$, which follows from the above equation for $\phi_2$ and the fact that $a_1 \in \tilde{W}_0$. Then, 

	\begin{align}
		\psi(p_0 , p_1 , p_2) & = \phi_1(\phi_2(p_0, p_2), p_1) \\
		& = P_{0}'\phi_2(p_0, p_2) + \bp{a_1, \phi_2(p_0, p_2)} (p_{1}-c_1) \\
		& = P_{0}'\tilde{P}_{0}p_0 + \bp{a_2, p_{0}} (p_{2}-c_2) + \bp{a_1, p_0} (p_{1}-c_1)
	\end{align}
	
	\noindent where $P_{0}'\tilde{P}_{0}$ is the orthogonal projector onto $\tilde{W}_0 \cap W_0 = \tilde{V}_0 \cap \spa{a_1,a_2}^\perp$. A similar calculation shows that $\psi$ satisfies \cref{eq:genPsiEqnI}, since $\phi_1$ and $\phi_2$ each satisfy it.
\end{example}

This procedure can be repeated as many times as necessary to obtain the more general warped products given by \cref{eq:genPsiEqnI}. Hence the properties of the more general warped product decompositions of $\eunn$ can be deduced from \cref{thm:WPDecompsI}. 

The following proposition shows that any proper warped product decomposition of $\eunn$ in canonical form restricts to a warped product decomposition of $\eunn(\kappa)$ where $\kappa \neq 0$. Its proof is straightforward consequence of \cref{eq:wpdecpSpI}.

\begin{theorem}[Restricting Warped products to $\eunn(\kappa)$] \label{thm:eunnKapRestWPI}
	Let $\psi$ be a proper warped product decomposition of $\eunn$ associated with $(\bar{p}; \bigobot\limits_{i=0}^{k} V_{i}; a_{1},...,a_{k})$ in canonical form. Suppose $\kappa^{-1} := \bar{p}^{2} \neq 0$ and let $N' := N_{0}(\kappa) \times_{\rho_{1}} N_{1} \times \cdots \times_{\rho_{k}} N_{k}$. Then $\phi : N' \rightarrow \eunn(\kappa)$ defined by $\phi := \psi|_{N'}$ is a warped product decomposition of $\eunn(\kappa)$ passing through $\bar{p}$.
\end{theorem}
\begin{proof}
	See \cref{thm:eunnKapRestWP}.
\end{proof}

Hence the details of warped product decompositions of $\eunn(\kappa)$ can be deduced from \cref{thm:WPDecompsI}. More information on these decompositions can be found in the following sections. In particular, see \cref{thm:wpInEunn,thm:wpInEunnKap}. Some examples can be found in \cite{Nolker1996} and also in a future article where we apply these results to construct coordinates which separate the Hamilton-Jacobi equation.

\section{pseudo-Riemannian Submanifolds and Foliations} \label{sec:pRiemSubFol}

In this section we will summarize the theory of pseudo-Riemannian submanifolds and foliations that will be useful to us. We can conveniently treat this as a special case of the theory of pseudo-Riemannian distributions, so we will present this first. For more details on pseudo-Riemannian submanifolds see (for example) \cite{barrett1983semi,Lee1997}. Similarly for pseudo-Riemannian foliations see \cite{rovenskii1998foliations,tondeur1988foliations}.

\subsection{Brief outline of The Theory of Pseudo-Riemannian Distributions }

The following brief exposition of the theory of pseudo-Riemannian distributions is a combination of that given in \cite{Meumertzheim1999} and \cite{Coll2006}. Suppose $E$ is an m-dimensional non-degenerate distribution defined on a pseudo-Riemannian manifold $\bar{M}$. Then we use the orthogonal splitting $T \bar{M} = E \obot E^{\perp}$, $V = V^E + V^{E^{\perp}}$, to define a tensor $s^E : T \bar{M} \times E \rightarrow E^{\perp}$ and a linear connection $\nabla^E$ for $E$ by:

\begin{equation}
	\nabla_XY = \nabla^E_XY + s^E(X,Y)
\end{equation}

\noindent for all $X \in \ve(\bar{M})$ and $Y \in \Gamma(E)$. $s^E$ is called the \emph{generalized second fundamental form} of $E$ and the above equation is referred to as the \emph{Gauss equation}. One can also check that $\nabla^E$ is metric compatible, i.e. $X\bp{Y,Z} = \bp{\nabla^E_XY,Z} + \bp{Y,\nabla^E_XZ}$ for all  $X \in \ve(\bar{M})$ and $Y,Z \in \Gamma(E)$.

For the remainder of the discussion we set $s^E := s^E|_{(E\times E)}$. For $X,Y \in \Gamma(E)$, we can further decompose $s^E(X,Y)$ into its anti-symmetric and symmetric parts 

\begin{align}
	s^E(X,Y) & = (\nabla_XY)^{E^{\perp}} = \frac{1}{2}(\nabla_XY + \nabla_YX)^{E^{\perp}} + \frac{1}{2}(\nabla_XY - \nabla_YX)^{E^{\perp}} \\
	& = h^E(X,Y) + A^E(X,Y)  \\
	A^E(X,Y) & := \frac{1}{2}(\nabla_XY - \nabla_YX)^{E^{\perp}} \\
	h^E(X,Y) & := \frac{1}{2}(\nabla_XY + \nabla_YX)^{E^{\perp}}
\end{align}

Since $\nabla$ is torsion-free, $A^E(X,Y) = \frac{1}{2}([X,Y])^{E^{\perp}}$, hence $E$ is integrable iff $A^E \equiv 0$. $h^E$ is called the \emph{second fundamental form} of $E$. The second fundamental form can be decomposed in terms of its trace to get a further classification of $E$ as follows:

\begin{align}
	h^E(X,Y) & = \bp{X,Y} H_E + h^E_T(X,Y) \\
	 H_E & = \frac{1}{m} \tr{h^E}
\end{align}

\noindent where $h^E_T$ is trace-less. $H_E$ is called the \emph{mean curvature normal} of $E$. $E$ is called \emph{minimal}, \emph{umbilical} or \emph{geodesic}\footnote{Note that some authors use the name auto-parallel instead \cite{Meumertzheim1999}.} if $s^E(X,Y) = h^E_T(X,Y)$, $s^E(X,Y) = \bp{X,Y} H_E$ or $s^E(X,Y) = 0$ respectively for all $X,Y \in \Gamma(E)$. We add the qualification ``almost'' to the three definitions above by replacing $s^E$ with $h^E$; this just drops the requirement that $A^E \equiv 0$. For example $E$ is \emph{almost umbilical} iff $h^E_T = 0$. We remark that when $E$ is one dimensional $h^E_T = 0$ trivially, hence all one dimensional non-degenerate foliations and all one dimensional pseudo-Riemannian submanifolds are trivially umbilical. If $E$ is umbilical and $\nabla_X^{E^{\perp}}H_E = 0$ for all $X \in \Gamma(E)$ then $E$ is called \emph{spherical}. Finally if $E$ is spherical and $E^{\perp}$ is geodesic then $E$ is called \emph{Killing}.

We also note here that $s^E$ and $s^{E^{\perp}}$ are not independent of each other:

\begin{proposition} \label{prop:sEsEperp}
	For $X,Y \in \Gamma(E)$ and $Z \in \Gamma(E^{\perp})$, the following holds:
	\begin{equation}
		\scalprod{s^E(X,Y)}{Z} = - \scalprod{Y}{s^{E^{\perp}}(X,Z)}
	\end{equation}
\end{proposition}
\begin{proof}
	\begin{align}
		0 & = \nabla_X\bp{Y,Z} \\
		  & = \bp{\nabla_XY,Z} + \bp{Y,\nabla_XZ} \\
		  & = \scalprod{s^E(X,Y)}{Z} + \scalprod{Y}{s^{E^{\perp}}(X,Z)}
	\end{align}
\end{proof}

\subsection{Specialization to pseudo-Riemannian submanifolds}

Suppose $\phi : M \rightarrow \bar{M}$ is a local embedding of (a pseudo-Riemannian submanifold) $M^m$ inside $\bar{M}^n$. Then for any point $p \in M$, it is known that there exist local coordinates $(x^i)$ on $\bar{M}$, such that the subset
\begin{equation}
	\{(x^1,\dotsc,x^m,x^{m+1}, \dotsc,x^n) : x^{m+1}=c_{m+1},\dotsc,x^n=c_n\}
\end{equation}
for some $c_{m+1},\dotsc,c_n \in \R$ can be identified with $\phi(U)$ where $U$ is an open subset with $p \in U \subseteq M$. These coordinates induce a local foliation $L$ in a neighborhood of $p$, with $M$ being a leaf given by the above equation. We will refer to such a foliation as a (local) foliation of $\bar{M}$ associated with $M$. Now suppose $L$ is an arbitrary foliation of $\bar{M}$ associated with $M$, and let $E$ be the induced distribution. Locally we can assume $L$ is a foliation by pseudo-Riemannian submanifolds of $\bar{M}$, hence $E$ is non-degenerate and the discussion in the previous section applies to it. Since $E$ is integrable, it follows that for any $X,Y \in \Gamma(E)$, that $[X,Y] \in \Gamma(E)$. Throughout this discussion, for any $X \in \Gamma(E)$, we let $\tilde{X} \in \ve(M)$ denote the unique vector field such that for any $p \in M$, we have $X_{\phi(p)} = \phi_* \tilde{X}_p$. Then for any $X,Y \in \Gamma(E)$ we see that

\begin{equation}
	[X,Y]|_{\phi(p)} =  \phi_* [\tilde{X},\tilde{Y}]|_p
\end{equation}

Thus $[X,Y]|_{\phi(p)}$ depends only on $[\tilde{X},\tilde{Y}]|_p$ in $M$. 

Now denote by $\nabla$ (resp. $\bar{\nabla}$) the Levi-Civita connection on $M$ (resp. $\bar{M}$). By the uniqueness properties of the Levi-Civita connection on $M$, it follows that for any $X,Y \in \Gamma(E)$ we have for any $p \in M$ that

\begin{equation}
	(\bar{\nabla}_X^E Y)|_{\phi(p)} =  \phi_* (\nabla_{\tilde{X}} \tilde{Y})|_p
\end{equation}

Thus $(\bar{\nabla}_X^E Y)|_{\phi(p)} $ depends only on $(\nabla_{\tilde{X}} \tilde{Y})|_p$ in $M$. By also using the Gauss equation, we observe that for any $p \in M$, that $(\bar{\nabla}_X Y)|_{\phi(p)} $ depends only on $\tilde{X}$ and $\tilde{Y}$.

In consequence of these observations, it follows that the theory presented for pseudo-Riemannian distributions induces a similar one for pseudo-Riemannian manifolds. We now connect this with the standard notations \cite{chen2011pseudo}; in effect this removes the appearance of the extraneous distribution, $E$.

In this case $s^E \equiv h^E$ and $h := (h^E)|_{M}$, then the Gauss equation becomes:

\begin{equation}
	\bar{\nabla}_X Y = \nabla_X Y + h(X,Y)
\end{equation}

\noindent for all $X,Y \in \ve(M)$. We denote the set of normal vector fields over $M$, i.e. the restriction of $\Gamma(E^{\perp})$ to $M$ by $\ve(M)^{\perp}$. The Gauss equation for $E^{\perp}$ is usually called the \emph{Weingarten equation} and is only defined for $X \in \ve(M)$ and $Y \in \ve(M)^{\perp}$. This is because in this case, $\bar{\nabla}_X Y$ depends only on the values that $X$ and $Y$ take on $M$\footnote{This is because for any $p \in \bar{M}$, $(\bar{\nabla}_X Y)|_p$ depends only on the values of $Y$ along any curve tangent to $X_p$. See Lemma~4.8 in \cite{Lee1997} and the following exercise, or Proposition~3.18~(3) in \cite{barrett1983semi}.}. Thus we can let  $A_Y(X) := -s^{E^{\perp}}(X,Y)$, $\nabla^{\perp}_XY := \bar{\nabla}^{E^{\perp}}_X Y$ and the Gauss equation (for $E^\perp$) becomes:

\begin{equation}
	\bar{\nabla}_X Y = \nabla^{\perp}_XY - A_Y(X)
\end{equation}

\noindent for all $X \in \ve(M)$ and $Y \in \ve(M)^{\perp}$. Note that the properties of $\bar{\nabla}^{E^{\perp}}$ imply that $\nabla^{\perp}$ is a connection\footnote{More precisely it satisfies the properties in definition~3.9 in \cite{barrett1983semi} and is metric compatible.} on $\ve(M)^{\perp}$. In this notation, the relationship between $s^E$ and $s^{E^\perp}$ given in \cref{prop:sEsEperp} becomes:

\begin{equation} \label{eq:secFwein}
	\bp{h(X,Y),Z} = \bp{A_Z(X),Y}
\end{equation}

\noindent for all $X,Y \in \ve(M)$ and $Z \in \ve(M)^{\perp}$.

Finally, we note that the definitions of minimal, umbilical, or geodesic foliations induces corresponding definitions for submanifolds. For example, a submanifold is geodesic if its second fundamental form vanishes identically.

In conclusion, we should mention that even though we have given a concise presentation of the theory, it's not useful for practical calculations. For these, one will have to evaluate these quantities in terms of curves on $M$. See for example, Proposition~4.8 in \cite{barrett1983semi}. 

%

\section{Circles and Spherical Submanifolds*} \label{sec:circSphSub}

%

In this section we will briefly overview the theory of circles and spherical submanifolds of pseudo-Riemannian manifolds. Circles are covariantly defined using the Frenet formula, but the definition of a sphere requires more work \cite{Nomizu1973}. This material is not necessary to read the rest of the article, although it gives an application of the general theory presented in the previous section, a geometric interpretation of spherical submanifolds, and gives some background for the results on the intrinsic properties of warped products to come. We also present this theory here because it's not covered in standard references, in contrast with the corresponding theory for geodesic submanifolds (see \cite{barrett1983semi}).

A proper circle\footnote{Sometimes these are called geodesic circles \cite{Aminova2003}. This name emphasizes the fact that we due not require the image of these curves to be a compact set, i.e. homeomorphic to $\Si^1$.} in a pseudo-Riemannian manifold is defined using the Frenet formula as a unit speed curve whose first curvature is constant and non-zero and remaining curvatures vanish. To be precise, let $\gamma(t)$ be a unit speed curve in $M$, i.e. $\dot{\gamma}^2 = \pm 1$. Let $X := \dot{\gamma}$. Let $\kappa(t) := \norm{\nabla_X X}$ be the (first) curvature of $\gamma$. Assuming $\kappa \neq 0$, we define $Y$ to be the unit vector field over $\gamma$ derived from $\nabla_X X$. I.e. $Y$ satisfies the following equation

\begin{equation}
	\nabla_X X = \kappa Y
\end{equation}

A proper circle is defined to be a curve which satisfies $\nabla_XY = c X$ for some $c \in \R \setminus \{0\}$. We observe that

\begin{align}
	\bp{\nabla_XY , X} & = - \bp{Y , \nabla_XX} \\
	& = -\kappa \bp{Y , Y}
\end{align}

The above equation implies that $c = -\varepsilon_0\varepsilon_1 \kappa$ where $\varepsilon_0 := \sgn \bp{X,X}$ and $\varepsilon_1 := \sgn \bp{Y,Y}$. Thus a proper circle is defined by the equations

\begin{align}
	\nabla_X X & = \kappa Y \\
	\nabla_XY & = - \varepsilon_0\varepsilon_1\kappa X
\end{align}

\noindent where $\kappa \neq 0$ is a constant. A proper circle satisfies the following third order ODE \cite{Abe1990}:

\begin{equation} \label{eq:circ}
	\nabla_X \nabla_X X = - \bp{\nabla_X X,\nabla_X X} \bp{X,X} X
\end{equation}

Conversely we will see shortly that any unit speed curve satisfying the above equation with $\bp{\nabla_X X,\nabla_X X} \neq 0$ is a proper circle. We define a \emph{circle} in a pseudo-Riemannian manifold to be a unit speed curve satisfying the above equation, hereafter called the \emph{circle equation}. The following lemma shows that any pseudo-Riemannian manifold admits circles:

\begin{lemma}[Existence and Uniqueness of Circles \cite{Nomizu1974}] \label{lem:circUniq}
	Consider the following initial conditions: $p \in M$, a unit vector $X_p \in T_p M$ and $Y_p \in X_p^{\perp}$. There exists a unique locally defined unit speed curve $\gamma(t)$ in $M$ satisfying \cref{eq:circ} and the initial conditions:
	\begin{align}
		\gamma(0) & = p \\
		\dot{\gamma}(0) & = X_p \\
		(\nabla_XX)|_p & = Y_p
	\end{align}
	
	\noindent where $X := \dot{\gamma}$ and $Y := \nabla_XX$. Furthermore, $\bp{Y,Y}$ is constant along any circle.
\end{lemma}
\begin{proof}
	It follows by the existence and uniqueness theorem for ODEs that there exists a unique locally defined curve $\gamma(t)$ satisfying \cref{eq:circ} with the above initial conditions. Then observe the following:
	
	\begin{align}
		\nabla_X \bp{X,X} & = 2 \bp{X,\nabla_XX} = \bp{X,Y} \\
		\nabla_X \bp{X,Y} & = \bp{Y,Y} + \bp{X,\nabla_X Y} \\
		& \overset{\eqref{eq:circ}}{=} \bp{Y,Y} - \bp{X,X}^2 \bp{Y,Y} \\
		& = \bp{Y,Y}(\bp{X,X}^2 - 1)
	\end{align}
	
	The above two equations define a system of ODEs for $\bp{X,X}$ and $\bp{X,Y}$, with initial values $\bp{X,X}|_p = \varepsilon = \pm 1$ and $\bp{X,Y}|_p = 0$. Thus by the uniqueness of the solutions, it follows that $\bp{X,X} = \varepsilon$ and $\bp{X,Y} = 0$ wherever $\gamma$ is defined. Hence $\gamma$ is a unit speed curve.
	
	Finally observe that
	
	\begin{align}
		\nabla_X\bp{Y,Y} & = 2 \bp{\nabla_XY, Y}  \\
		& \overset{\eqref{eq:circ}}{=} - 2 \bp{\nabla_X X,\nabla_X X} \bp{X,X} \bp{X,Y} \\
		& = 0
	\end{align}
	
	Hence $\bp{Y,Y}$ is constant.
\end{proof}

Note that $k := \norm{Y}$ in the above lemma is usually called the \emph{curvature} of the circle. In Riemannian manifolds, circles are completely classified by their curvature, although this is not true for pseudo-Riemannian manifolds. Using the above lemma we can classify circles in a pseudo-Riemannian manifold as follows. Let $\gamma(t)$ be a circle in $M$ and suppose $\gamma$ satisfies the initial conditions of the above lemma. Then $\gamma$ can be classified as follows depending on $Y_p$:

\begin{description}
	\item[Geodesic: ] If $Y_p= 0$.
	\item[Proper Circle: ] If $\bp{Y,Y}|_p \neq 0$.
	\item[Null Circle: ] If $\bp{Y,Y}|_p = 0$ but $Y_p \neq 0$, i.e. $Y_p$ is lightlike, hence \cref{eq:circ} reduces to $\nabla_X\nabla_X X = 0$.
\end{description}

Note that this classification is well defined globally since $\bp{Y,Y}$ is a constant of a circle and the uniqueness theorem for ODEs forces any circle with $Y_p = 0$ to be a geodesic.

\begin{example}[Geodesics in Spherical Submanifolds \cite{Kassabov2010}] \label{ex:geoInSph}
	Let $M$ be a spherical submanifold of $\bar{M}$. Suppose $\gamma(t)$ is a unit speed geodesic on $M$. We will show that $\gamma$ is a circle in $\bar{M}$. By the Gauss equation, we have the following:
	
	\begin{equation}
		\bar{\nabla}_X X = \bp{X,X} H
	\end{equation}
	
	Then by the Weingarten equation and using the fact that $\bar{\nabla}^{\perp} H = 0$ where $\bar{\nabla}^{\perp}$ is the induced normal connection over $M$, we have the following:
	
	\begin{align}
		\bar{\nabla}_X \bar{\nabla}_X X & = \bp{X,X} \bar{\nabla}_X H \\
		& = -\bp{X,X} A_H(X) \\
		& = -\bp{X,X} \bp{H,H} X \\
		& = -\bp{X,X}\bp{\bar{\nabla}_X X,\bar{\nabla}_X X} X
	\end{align}
	
	\noindent since for any $Z \in \ve(M)^{\perp}$, $\bp{A_H(X),Z} \overset{\eqref{eq:secFwein}}{=} \bp{h(X,Z),H} = \bp{X,Z} \bp{H,H}$.
\end{example}

We note here that the above example in combination with \cref{lem:circUniq} shows that the mean curvature vector field of a spherical submanifold is locally determined by its value at a single point. In \cref{ex:eunnPropCirc} we will describe the circles in pseudo-Euclidean space after we have described the spherical submanifolds of the space.

We will now present some additional results that show how circles can be used to characterize spherical submanifolds. These results were first obtained for the Riemannian case by \citeauthor{Nomizu1974} in \cite{Nomizu1974}. They were generalized to the Lorentzian case by \citeauthor{Ikawa1985} in \cite{Ikawa1985} and to the pseudo-Riemannian case by \citeauthor{Abe1990} in \cite{Abe1990}.

For the following theorems we denote a pseudo-Riemannian manifold $M$ with signature $\alpha$ by $M_{\alpha}$. The following theorem characterizes spherical submanifolds in terms of circles, it is analogous to the corresponding theorem for geodesics and geodesic submanifolds (see \cite[section~4.4]{barrett1983semi}).

\begin{theorem}[Circles and Spheres \cite{Abe1990}] \label{thm:circNdSph}
	Let $M_{\alpha}$ be an $n$ dimensional pseudo-Riemannian submanifold of $\bar{M}_{\beta}$. For any $\varepsilon_0 \in \{-1,1\}$ and $\varepsilon_1 \in \{-1,0,1\}$ satisfying $2 - 2\alpha \leq \varepsilon_0 + \varepsilon_1 \leq 2 n - 2 \alpha -2$ and $k \in \R^+$, the following are equivalent:
	\begin{enumerate}[(a)]
		\item Every circle in $M_{\alpha}$ with $\bp{X,X} = \varepsilon_0$ and $\bp{\nabla_X X,\nabla_X X} = \varepsilon_1 k^2$ is a circle in $\bar{M}_{\beta}$.
		\item $M_{\alpha}$ is a spherical submanifold of $\bar{M}_{\beta}$.
	\end{enumerate} 
\end{theorem}
\begin{proof}
%
%
%
%
%
See \cite{Abe1990}.
\end{proof}

More intuitively, the above theorem states that a spherical submanifold $M$ is precisely a submanifold in which all circles in $M$ are circles in the ambient space. Also note that the above theorem shows that a circle is precisely a spherical submanifold of dimension one. The following theorem is a variant of the above theorem which is known to hold (in full generality) only in the strictly pseudo-Riemannian case.

\begin{theorem}[Circles and Spheres II \cite{Abe1990}] 
	Let $M_{\alpha}$ be an $n$ dimensional ($1 \leq \alpha \leq n - 1$) pseudo-Riemannian submanifold of $\bar{M}_{\alpha}$ having the same signature $\alpha$. For any $\varepsilon_0 \in \{-1,1\}$, the following are equivalent:
	\begin{enumerate}[(a)]
		\item Every geodesic in $M_{\alpha}$ with $\bp{X,X} = \varepsilon_0$ is a circle in $\bar{M}_{\alpha}$.
		\item $M_{\alpha}$ is a spherical submanifold of $\bar{M}_{\alpha}$.
	\end{enumerate}
\end{theorem}

These results can be further generalized by considering more general types of curves such as helices (which we will not define here). See \cite{Nakanishi1988} where a theorem analogous to \cref{thm:circNdSph} is proven characterizing helices in terms of geodesic submanifolds. Also in \cite{Jun1994} results relating conformal circles to umbilical submanifolds are presented.

The following lemma describes how much information is required to specify a sphere. It is a partial generalization of the corresponding lemma for the Riemannian case proven in \cite{Kassabov2010}.

\begin{lemma}[Uniqueness of Spheres] \label{lem:unqSph}
	Suppose that $M$ and $N$ are connected and geodesically complete spherical submanifolds of $\bar{M}$ both satisfying the following condition: For some $p \in M \cap N$, $M$ and $N$ are tangent and have the same mean curvature vectors. Then $M \equiv N$.
\end{lemma}
\begin{proof}
	Our proof is a generalization of the proof of lemma~4.14 in \cite[P.~105]{barrett1983semi}.

	Let $q \in M$ be arbitrary and suppose that $\gamma(t)$ is a geodesic segment in $M$ running from $p$ to $q$. Then observe that $\gamma$ is a geodesic circle in $\bar{M}$ with velocity $X_p$ and acceleration $\bp{X,X}|_p H^M_p$ at $p$ where $H^M$ is the mean curvature vector field of $M$. By the uniqueness of circles (see \cref{{lem:circUniq}}) and the hypothesis it follows that $\gamma$ is also geodesic in $N$ which is defined everywhere since $N$ is geodesically complete. Note that this implies that mean curvature vector fields of $M$ and $N$ coincide over $\gamma$, so we denote this vector field by $H$.

	Now suppose $Z_p \in T_p M \cap X_p^{\perp}$ and let $Z$ be the parallel transport of $Z_p$ over $\gamma$ with respect to $M$. Since parallel transport is an isometry, $\bp{Z, X} = 0$. Thus by the Gauss equation, 
	\begin{align}
		\bar{\nabla}_X Z & = \nabla^M_X Z + \bp{Z, X}H \\
		& = 0
	\end{align}
	
	\noindent where $\bar{\nabla}$ is the Levi-Civita connection on $\bar{M}$ and $\nabla^M$ is the induced Levi-Civita connection on $M$. Thus $Z$ is also the parallel transport of $Z_p$ over $\gamma$ with respect to $\bar{M}$.
	
	Thus the parallel transport of $T_p M \cap X_p^{\perp}$ to $q$ on $\bar{M}$ is equal to $T_q M \cap X_q^{\perp}$. Similarly the parallel transport of $T_p N \cap X_p^{\perp}$ to $q$ on $\bar{M}$ is equal to $T_q N \cap X_q$. Since the parallel transport on $\bar{M}$ is uniquely determined, we deduce that $T_q M \cap X_q^{\perp} = T_q N \cap X_q^{\perp}$. Since $X_q \in T_q M,T_q N$, we conclude that $T_q M  = T_q N $. Thus since $M$ is connected, one can apply this argument to an arbitrary broken geodesic (see \cite{barrett1983semi}) to conclude that $M \subseteq N$.
	
	Finally by applying the argument for $M$ interchanged with $N$, we see that $M \equiv N$.
\end{proof}

Let $M$ be a space of constant curvature. We will show in this article that for every $p \in M$, non-degenerate subspace $V \subset T_p M$, and normal vector $H \in (T_p M)^\perp$ there exists a connected and geodesically complete spherical submanifold passing through $p$ with tangent space $V$ and mean curvature vector $H$ at $p$. In the following theorem, we will show that this property characterizes Riemannian spaces of constant curvature. For the following theorem, we say a Riemannian manifold $M$ satisfies the \emph{axiom of $r$-spheres} if: for every $p \in M$ and any $r$ dimensional subspace $V \subset T_p M$ there exists a spherical submanifold passing through $p$ and tangent to $V$.

\begin{theorem}[Spheres in spaces of constant curvature \cite{leung1971}] \label{thm:sphSCC}
	Let $M$ be a Riemannian manifold with dimension $n \geq 3$ and fix $2 \leq r < n$. Then $M$ is a space of constant curvature iff it satisfies the axiom of $r$-spheres (see above).
\end{theorem}
\begin{proof}
	See \cite{leung1971}.
\end{proof}

\section{Spherical Submanifolds of Spaces of Constant Curvature} \label{sec:SphSubSCC}

In this section $\kappa$ is allowed to be zero. The following optional proposition relates umbilical submanifolds to spherical ones in spaces of constant curvature.

\begin{proposition} \label{prop:sccUmbImpSph}
	Any umbilical submanifold of $\eunn(\kappa)$ with dimension greater than one is necessarily spherical.
\end{proposition}
\begin{proof}
	This follows from Lemma~3.2~(a) in \cite{chen2011pseudo}.
\end{proof}

Here we state some properties of spherical submanifolds in spaces of constant curvature.

\begin{proposition}[Spherical Submanifolds in Spaces of Constant Curvature] \label{prop:eunnUmb}
	Let $\phi : N \rightarrow \eunn(\kappa)^\circ$ be an isometric immersion of a pseudo-Riemannian manifold $N$. If $N$ is a spherical submanifold, then
	
	\begin{enumerate}[(a)]
		\item $\bp{H,H}$ is constant. 
		\item $N$ is of constant curvature $\kappa + \bp{H,H}$ \label{it:umbCc}
	\end{enumerate}
\end{proposition}
\begin{proof}
	Lemma~3.2 from \cite{chen2011pseudo}.
\end{proof}

\section{Standard spherical submanifolds of pseudo-Euclidean space} \label{sec:StdSubEunn}

We collect some properties of $\eunn(\kappa)$ in the following proposition.

\begin{proposition} \label{prop:eunnStdHQuad}
	Let $r$ denote the dilatational vector field and $r^2 = \bp{r,r}$. Fix $r^2 \in \R$, the following are true about $\eunn(\frac{1}{r^2})$
	\begin{enumerate}[(a)]
		\item It is a spherical submanifold with mean curvature normal
		\begin{equation} \label{eq:eunnkapMean}
			H = - \frac{r}{r^2}
		\end{equation}
		\item It has constant curvature $\dfrac{1}{r^2}$ and is geodesically complete.
	\end{enumerate}
\end{proposition}
\begin{proof}
	The first follows from \cite[Lemma~4.27]{barrett1983semi}. When $\dim \eunn(\frac{1}{r^2}) > 1$, the first result together with \cref{prop:sccUmbImpSph} shows that $\eunn(\frac{1}{r^2})$ is a spherical submanifold. In any case, it follows from \cref{eq:eunnkapMean} that $\eunn(\frac{1}{r^2})$ is a spherical submanifold. Hence the second result follows from \cref{prop:eunnUmb}~\ref{it:umbCc}. It follows from lemma~4.29 in \cite{barrett1983semi} that $\eunn(\frac{1}{r^2})$ is geodesically complete.
\end{proof}

We collect similar properties of $\punn$.

\begin{proposition} \label{prop:eunnPunn}
	The following are true about $\punn$ with mean curvature vector $-a$:
	\begin{enumerate}[(a)]
		\item It is a spherical submanifold with mean curvature normal
		\begin{equation}
			H = -a
		\end{equation}
		\item It is globally isometric to $\eunn$.
	\end{enumerate}
\end{proposition}
\begin{proof}
	Consider the map $\psi$ given by \cref{eq:psiPunn}. It then follows that for $v \in T V$, 
	
	\begin{align}
		\psi_*v = v - \bp{v,x} a
	\end{align}
	
	The above equation shows that the induced metric at each point is the induced metric on $V$. Hence $\punn$ is globally isometric to $\eunn$. Now to calculate the second fundamental form, we have for $w,v \in T V$:
	
	\begin{align}
		\nabla_{\psi_*w} \psi_*v & = \nabla_w v - \bp{\nabla_w v, x} a - \bp{v, w} a \\
		& = \psi_* \nabla_w v - \bp{v, w} a
	\end{align}
	
	\noindent Hence it follows that $\punn$ is umbilical with mean curvature vector $-a$. Since $-a$ is covariantly constant, it follows that $\punn$ is spherical.
\end{proof}

\section{Warped Products} \label{sec:WP}

In this section we define the warped product and give some properties of it which will be used in this article. The content of this section is primarily from \cite{Meumertzheim1999} where the more general notion of a twisted product was introduced. For more on warped products see \cite{Meumertzheim1999,Zeghib2011}. A warped product can be defined to be a special case of a twisted product as follows:

\begin{definition}[Twisted and Warped Products] \label{def:twistedProd}
	Let $M = \prod_{i=0}^k M_i$ be a product of pseudo-Riemannian manifolds $(M_i, g_i)$ where $\dim M_i > 0$ for $i > 0$. Suppose for $i=0,...,k$, $\pi_i : M \rightarrow M_i$ is the projection map and $\rho_i: M \rightarrow \R^+$ is a function. The following metric $g$ on $M$ is called a \emph{twisted product metric}
	
	\begin{equation}
		g(X,Y) = \sum_{i=0}^k \rho_i^2 g_i(\pi_{i *}X,\pi_{i *}Y) \quad \text{for $X,Y \in \ve(M)$}
	\end{equation}
	
	If each $\rho_i$ depends only on $M_0$ and $\rho_0 \equiv 1$ then $g$ is called a \emph{warped product metric} and $(M,g)$ is called a \emph{warped product}. The warped product is denoted by $M_0 \times_{\rho_1} M_1 \times \cdots \times_{\rho_k} M_k $. $M_0$ is called the geodesic factor of the warped product and the $M_i$ for $i > 0$ are called spherical factors.
\end{definition}

\begin{example}[Prototypical warped product]
	The prototypical example of a warped product is the following warped product defined in (an open subset of) $\E^n$, which is the product manifold $\R^+ \times S^{n-1}$ equipped with the metric $g = \d \rho^2 + \rho^2 \tilde{g}$ where $\tilde{g}$ is the metric of the $(n-1)$-sphere $S^{n-1}$.
\end{example}

In the above example we say the warped product $\R^+ \times_{\rho} S^{n-1}$ is a \emph{warped product decomposition} of $\E^n$. In general a \emph{warped product decomposition} of a given pseudo-Riemannian manifold $M$ is a warped product which is (locally) isometric to $M$.

Each factor $M_i$ of the product manifold induces a foliation $L_i$ of $M$. For any $\bar{p} \in M$ the leaf of this foliation through $\bar{p}$, $L_i(\bar{p})$, is given as follows:

\begin{align}
	L_i(\bar{p}) &: = \{p \in M :  p = (\bar{p}_1,\dotsc, \bar{p}_{i-1}, p_i, \bar{p}_{i+1},\dotsc, \bar{p}_k), \ p_i \in M_i  \}
\end{align}

\noindent where $\bar{p}_j = \pi_j(\bar{p})$. We let $E_i$ denote the integrable distribution induced by $L_i$, then $\Ei = (E_i)_{i=1}^k$ is called the \emph{product net} of $\prod_{i=1}^k M_i$.

We also note here that a warping function $\rho_i$ of a warped product is only uniquely defined modulo products of constants. To elaborate, from the above definition one sees that we can multiply $\rho_i^2$ by any $c \in \R^+$ if we divide $g_i$ by $c$. The geometry of the warped product is not altered by such transformations as we will see. We say that the warping functions are \emph{normalized} (with respect to a point $\bar{p} \in M$), if for each $i$, $\rho_i(p) = 1$ for all $p \in L_i(\bar{p})$.

We record here some properties of the warped product:

\begin{proposition}[Properties of the Warped Product \cite{Meumertzheim1999}] \label{prop:tpProps}
	Let $\sideset{^{\rho}}{_{i=0}^k}{\prod}M_i$ be a warped product with product net $\Ei = (E_i)_{i=0}^k$.
	
	\begin{enumerate}
		\item $\Ei$ is orthogonal, i.e. it satisfies: $T M = \bigobot\limits_{i=0}^k E_i$
		\item For each $i > 0$ the distribution $E_i$ is Killing with mean curvature normal $H_i =  - \nabla (\log \rho_i)$. Hence $E_0$ is geodesic.
	\end{enumerate}
\end{proposition}
\begin{proof}
	See proposition~2 in \cite{Meumertzheim1999}.
\end{proof}

The following theorem gives a converse to the above proposition.

\begin{theorem}[Geometric Characterization of Warped Products \cite{Meumertzheim1999}] \label{thm:geomTPWP}
	Let $M = \prod_{i=0}^kM_i$ be a connected product manifold equipped with metric $g$ and orthogonal product net $\Ei = (E_i)_{i=0}^k$. Then $g$ is the metric of a warped product $M_0 \times_{\rho_1} M_1 \times \cdots \times_{\rho_k} M_k $ iff $E_i$ are Killing foliations for $i = 1,...,k$.
\end{theorem}

In particular one should note that for a warped product, the manifolds $M_i$ are spherical submanifolds of $M$. This is an important observation in constructing warped products. We say a warped product is \emph{proper} if none of the spherical factors are geodesic submanifolds. One can check (see proposition~2 in \cite{Meumertzheim1999}) that this is equivalent to requiring $\rho_i$ to be non-constant.

The above theorem also shows that any manifold can be made into a spherical submanifold (of some other manifold). Hence this notion is only interesting with respect to a fixed manifold.

\section{Warped product decompositions of Spaces of Constant Curvature} \label{sec:wpSCC}

In this section we study warped product decompositions of $\eunn(\kappa)$ where $\kappa$ may equal zero. Let $M = M_0 \times_{\rho_1} M_1 \times \cdots \times_{\rho_k} M_k$ be a warped product and $\psi : M \rightarrow \eunn(\kappa)$ a warped product decomposition of  $\eunn(\kappa)$.Fix $\bar{p} \in \psi(M)$. Let $H_i =  - \nabla (\log \rho_i)$ be the mean curvature vector field associated to the canonical foliation $L_i$ generated by $M_i$ (see \cref{prop:tpProps}). Let $V_i := T_{\bar{p}_i} M_i$ for each $i$ and $z_i := H_i|_{\bar{p}} \in V_0$ for $i > 0$. Then note that

\begin{equation}
	T_{\overline{p}} M = \bigobot\limits_{i=0}^k V_i
\end{equation}

It follows from corollary~2 in \cite{Meumertzheim1999} that the mean curvature vectors satisfy the following equation for $i \neq j$:

\begin{equation} \label{eq:wpSccMean}
	\bp{z_i, z_j} = - \kappa
\end{equation}

In this case we say that $\psi$ is a warped product decomposition of $\eunn(\kappa)$ associated with the initial data $(\bar{p};\bigobot\limits_{i=0}^k V_i; a_1,...,a_k)$ where $a_i := \kappa\bar{p} - z_i$.

Conversely, let $\overline{p} \in \eunn(\kappa)$ where $n \geq 2$ and consider the following decomposition of $T_{\overline{p}} \eunn(\kappa)$, $T_{\overline{p}} \eunn = \bigobot\limits_{i=0}^k V_i$ into non-trivial subspaces (hence non-degenerate) with $k \geq 1$. Suppose $z_1,...,z_k \in V_0$ satisfy \cref{eq:wpSccMean}. Let $a_i := \kappa\bar{p} - z_i$ and assume additionally that the subset of non-zero $a_i$ are linearly independent. In this case, we say that $(\bar{p};\bigobot\limits_{i=0}^k V_i; a_1,...,a_k)$ are \emph{initial data} for a warped product decomposition of $\eunn(\kappa)$. We will show in this article that in a space of constant curvature there always exists a warped product decomposition associated with any given initial data. It follows from \cref{thm:sphSCC} that in the category of Riemannian manifolds with $n > 2$, this property characterizes spaces of constant curvature.

The additional condition requiring the $a_i$ to be linearly independent trivially holds in Euclidean space and in motivating applications. The reason we make this assumption will become more apparent later. Here is an optional lemma, which is given for completeness, and hints at why we make this assumption.

\begin{lemma}
	Suppose $a_1,\dotsc,a_k$ are linearly independent pair-wise orthogonal lightlike vectors. Then there exist vectors $b_1,\dotsc,b_k$ such that $\bp{a_i , b_j} = \delta_{ij}$ and $\bp{b_i , b_j} = 0$.
\end{lemma}
\begin{proof}
	Suppose to the contrary that for any $b_1$ satisfying $\bp{b_1, a_i} = 0$ for $i > 1$ we have $\bp{b_1, a_1} = 0$. Thus
	
	\begin{equation}
		\cap_{i=2}^k a_i^\perp \subseteq a_1^\perp
	\end{equation}
	
	Define $T : V \rightarrow \R^k$ by:
	
	\begin{equation}
		T(v) = (\bp{a_1 , v},\dotsc,\bp{a_k , v})
	\end{equation}
	
	By hypothesis we have $\dim \ker T \geq n - (k-1)$, hence $\dim \Ima T \leq k-1$ by the rank-nullity theorem. Thus $a_1^\flat,\dotsc,a_k^\flat$ are linearly dependent, a contradiction.
	
	Thus there exists $b_1 \in \cap_{i=2}^k a_i^\perp$ with $\bp{a_1,b_1} = 1$. The result then follows by induction. Indeed the next step is to find $b_2$ by applying the above result to $\{a_2,\dotsc,a_k\} \subset \spa{a_1,b_1}^\perp$ making use of the fact that $\spa{a_1,b_1}$ is non-degenerate by construction.
\end{proof}

It has been shown by Nolker in \cite{Nolker1996} that given any initial data for Riemannian spaces of constant curvature, there exists a unique warped product decomposition associated with the initial data. In this article we will show that given any initial data for a WP-decomposition of $\eunn(\kappa)$, there exists a WP-decomposition associated with the initial data. This WP-decomposition is probably uniquely determined but we don't use or prove this supposition.

One can also deduce that from corollary~2 in \cite{Meumertzheim1999} that the Hessian $H$ of each warping function $\rho$ of a space of constant curvature satisfies the following equation on the geodesic factor:

\begin{equation}
	H(X,Y) = -\kappa \rho \bp{X,Y}
\end{equation}

This proves the following fact:

\begin{lemma} \label{lem:EunnKapProdDecomp}
	A space of constant non-zero curvature does not admit product decompositions.
\end{lemma}

\section{Warped product decompositions of pseudo-Euclidean space} \label{sec:WPdecompEunn}

\subsection{Spherical submanifolds of pseudo-Euclidean space}

We first describe the spherical submanifolds of pseudo-Euclidean space. The following theorem is a generalization of Lemma~5 in \cite{Nolker1996} to pseudo-Euclidean space.

\begin{theorem}[Spherical submanifolds of $\eunn$] \label{thm:spherSub}
	Let $\overline{p} \in \glssymbol{eunn}$ be arbitrary, $V \subseteq \eunn$ a non-degenerate subspace with $m := \dim V \geq 1$, $\mu := \ind V$ and $z \in V^{\perp}$. Let $\tilde{\kappa} := z^2$, $a := -z$ and $W = \R a \obot V$. There is exactly one $m$-dimensional connected and geodesically complete spherical submanifold $\tilde{N}$ with $\overline{p} \in \tilde{N}$, $T_{\overline{p}} \tilde{N} = V$ and having mean curvature vector at $\overline{p}$, z. $\tilde{N}$ is an open submanifold of N; N is referred to as the spherical submanifold determined by $(\overline{p},V,a)$ and is given as follows (where $\simeq$ means isometric to):
	
	\begin{enumerate}[(a),style=multiline]
		\item \label{it:sphEunna} $a = 0$ iff N is geodesic, in this case $N \simeq \E^m_{\mu}$ \\
		\begin{equation}
			N = \overline{p} + V
		\end{equation}
		\item \label{it:sphEunnb} $a$ is timelike, then $\mu \leq \nu - 1$ and $N \simeq H^m_{\mu}(\tilde{\kappa})$
		\item \label{it:sphEunnc} $a$ is spacelike, then $N \simeq S^m_{\mu}(\tilde{\kappa})$ \\
		For cases (b) and (c), let $c = \overline{p} -\frac{a}{\tilde{\kappa}}$ be the center of N, then N is given as follows:
		\begin{equation}
			N = c + \{ p \in W \: | \: p^2 = \frac{1}{\tilde{\kappa}} \}
		\end{equation}
		\item \label{it:sphEunnd} $a$ is lightlike, then $\mu \leq \nu - 1$ and $N \simeq E^m_{\mu}$
		\begin{equation}
			N = \overline{p} + \{ p - \frac{1}{2} p^2 a  \: | \: p \in V \}
		\end{equation}
	\end{enumerate}
\end{theorem}
\begin{remark} \label{rem:tildN}
	$\tilde{N} = N$ except in the following two cases (which are anti-isometric): When $N \simeq H^m_0(\tilde{\kappa})$ or $N \simeq S^m_m(\tilde{\kappa})$, N is disconnected \cite[Section~4.6]{barrett1983semi} and so $\tilde{N}$ is given as follows:
	\begin{equation}
		\tilde{N} = N \cap (c+ \{p \in W \: | \: \bp{a,p} > 0 \})
	\end{equation}
\end{remark}
\begin{proof}
	First we note that it suffices to show that there exists a single connected and geodesically complete sphere satisfying the initial conditions. By \cref{lem:unqSph}, it must be unique.

	\Cref{it:sphEunna} is clear. For \cref{it:sphEunnb,it:sphEunnc}, it follows from \cref{prop:eunnStdHQuad} that $N$ is a sphere and the initial conditions are easily checked. The connectedness properties follow from lemma~4.25 in \cite{barrett1983semi}. It follows from lemma~4.29 in \cite{barrett1983semi} that $N$ is geodesically complete.
	
	\Cref{it:sphEunnd} follows from \cref{prop:eunnPunn}.
\end{proof}
\begin{remark}
	See \cite{chen2011pseudo} for a different proof.
\end{remark}

Since circles are one dimensional spherical submanifolds, we can use the above theorem to describe the circles in pseudo-Euclidean space.

\begin{example}[Proper Circles in pseudo-Euclidean space] \label{ex:eunnPropCirc} 
	Suppose $(\bar{p},\bar{V},k \bar{Y})$ are initial conditions for a proper circle as in \cref{lem:circUniq} with $\varepsilon_0 := \bar{V}^2 = \pm 1$, $\varepsilon_1 := \bar{Y}^2 = \pm 1$ and $\norm{k\bar{Y}} \neq 0$. We now describe the circle determined by this data.
	
	By \cref{ex:geoInSph} the proper circle determined by these initial conditions determine a spherical submanifold of $\eunn$ characterized by $(\bar{p},\R \bar{V},  \varepsilon_0 k \bar{Y})$. Now let $H := \varepsilon_0 k \bar{Y}$, $\kappa := \bp{H,H} = \varepsilon_1 k^2$ and $c := \overline{p} + \frac{H}{\kappa} = \bar{p} - \frac{\varepsilon_0 \varepsilon_1 \bar{Y} }{k}$.
	
	\begin{parts}
		\item Euclidean circle, $\gamma = \Si^1$: $\varepsilon_0 = \varepsilon_1 = \pm 1$ \\
		
		\begin{equation}
			\gamma(t) = c + \frac{1}{k}(\sin(k t) \bar{V} - \cos(k t) \bar{Y})
		\end{equation}
		
		\item Hyperbolic circle, $\gamma = H^1$: $\varepsilon_0 = 1,  \varepsilon_1 = - 1$
		\item de Sitter circle, $\gamma = \Si^1_1$: $\varepsilon_0 = - 1,  \varepsilon_1 =  1$ \\
		In the last two cases (which are anti-isometric), $\gamma$ is given as follows: 
		
		\begin{equation}
			\gamma(t) = c + \frac{1}{k}(\sinh(k t) \bar{V} - \varepsilon_0\varepsilon_1 \cosh(k t) \bar{Y})
		\end{equation}
	\end{parts}
\end{example}

One can give a similar example for geodesics and null circles. 

%

\subsection{Warped product decompositions of pseudo-Euclidean space}

Our classification of the warped product decompositions of $\eunn$ is based on the fact that a specification of the tangent spaces and mean curvature normals of the spherical foliations of a warped product at one point $\overline{p}$, uniquely determines a warped product decomposition in a neighborhood of $\overline{p}$. We now carry out this classification as follows. Suppose $\psi : N_0 \times_{\rho_1} N_1 \times \cdots \times_{\rho_k} N_k \rightarrow \eunn$ is a warped product decomposition of $\eunn$ associated with initial data $(\bar{p};\bigobot\limits_{i=0}^k V_i; -z_1,...,-z_k)$. By \cref{eq:wpSccMean}, the mean curvature vectors at $\bar{p}$ satisfy the following equation:

\begin{equation}
	\bp{z_i,z_j} = 0 \quad i \neq j
\end{equation}

We now only consider the case $\nu \leq 1$ as the other signatures are straightforward generalizations of these standard ones. In this case, we will use \cref{thm:spherSub} to classify $N_i$ up to homothety as follows. Say $z_1,...,z_l = 0$ and the remaining are non-zero, then for $i = 1,...,l$ the $N_i$ are pair-wise orthogonal planes passing through $\overline{p}$. We now consider the remaining possibilities:

\begin{parts}
	\item Since the $z_i$ are orthogonal, there is at most one lightlike direction, say $z_{l+1}$. The remaining lightlike $z_i$ are proportional to $z_{l+1}$, but since we assume the non-zero $z_i$ are linearly independent, we will work with only one lightlike vector $z_{l+1}$. Then $N_{l+1}$ a paraboloid isometric to Euclidean space. The orthogonality relations force the remaining $z_i$ to be space-like and hence the remaining $N_i$ are Euclidean spheres.
	\item Similarly, at most one of the $z_i$ can be timelike, say $z_{l+1}$. Then $N_{l+1}$ is isometric to hyperbolic space. The orthogonality relations force the remaining $z_i$ to be space-like and hence the remaining $N_i$ are Euclidean spheres.
	\item The remaining $z_i$ are spacelike. If $\ind V_0 = 1$ or $\ind V_0 = 0$ in Euclidean space, then the remaining $N_i$ are Euclidean spheres. If $\ind V_0 = 0$ in Minkowski space, then $\ind V_j = 1$ for precisely one $j \geq 1$, then $N_j$ is de Sitter space while the remaining $N_i$ are Euclidean spheres.
	\item All $z_i$ are zero. Then each $N_i$ is an affine plane and the warped product is a product of planes.
\end{parts}

We summarize our findings in the following theorem.

\begin{theorem}[Warped products in $\E^n$ and $M^n$] \label{thm:wpInEunn}
	Suppose $N = N_0 \times_{\rho_1} N_1 \times \cdots \times_{\rho_k} N_k$ is a proper warped product decomposition of an open subset of $\eunn$. If at most one of the $N_i$ are intrinsically flat, then N is isometric to one of the following warped products:

	If $\eunn$ is Euclidean space:
	\begin{align}
		\E^m \times _{\rho_1} S^{n_1} \times \cdots \times _{\rho_s} S^{n_s}
	\end{align}
	
	If $\eunn$ is Minkowski space:
	\begin{align}
		M^m \times _{\lambda_1} \E^{n_1} \times _{\rho_2} S^{n_2} \times \cdots \times _{\rho_s} S^{n_s} \\
		M^m \times _{\tau_1} H^{n_1} \times _{\rho_2} S^{n_2} \times \cdots \times _{\rho_s} S^{n_s} \\
		\E^m \times _{\rho_1} dS^{n_1} \times _{\rho_2} S^{n_2} \times \cdots \times _{\rho_s} S^{n_s} \\
		M^m \times _{\rho_1} S^{n_1} \times _{\rho_2} S^{n_2} \times \cdots \times _{\rho_s} S^{n_s}
	\end{align}
	
	\noindent where $\nabla \rho_i$,$\nabla \tau_i$,$\nabla \lambda_i$ is a spacelike,timelike, lightlike vector field respectively.
\end{theorem}

The above theorem shows that there are at 1 and 4 distinct types of proper singly warped products in Euclidean and Minkowski space respectively. One can show that the multiply warped products can be built up from the singly warped products by iteratively decomposing the geodesic factor of the warped product into another warped product which is ``compatible'' with the original. Thus we only describe a special subset of warped products for simplicity.

The following theorem describes this interesting class of warped products. Its proof can be deduced from Theorem~7 in \cite{Nolker1996}. It is a generalization of that theorem to pseudo-Euclidean space.

\begin{theorem}[ Standard Warped Products in $\eunn$ \cite{Nolker1996}] \label{thm:WPDecomps}
	Fix $\overline{p} \in \eunn$ where $n \geq 2$ and the following decomposition of $T_{\overline{p}} \eunn$, $T_{\overline{p}} \eunn = \bigobot\limits_{i=0}^k V_i$ into non-trivial subspaces (hence non-degenerate) with $k \geq 1$. Suppose $a_1,...,a_k \in V_0$ are pair-wise orthogonal. Let $\kappa_i := a_i^2$ and $\epsilon_i := \sgn \kappa_i$. We consider the following warped decompositions:
	
	\begin{description}
		\item[non-null warped decomposition] Let $\mu \geq 0$
		\begin{equation}
			\kappa_1 \leq \cdots \leq \kappa_{\mu} < 0 < \kappa_{\mu+1} \leq \cdots \leq \kappa_k
		\end{equation}
		
		\noindent In this case, let $c = \overline{p} - \sum\limits_{i=1}^k \frac{a_i}{\kappa_i}$ and $c_i = \overline{p} - \frac{a_i}{\kappa_i}$ for each $i = 1,...,k$.
		\item[null warped decomposition] $k = 1$, $a_1 := a$, $\kappa_1 =  a^2 = 0$ but $a \neq 0$, i.e. $a$ is lightlike.
		
		\noindent In this case, fix a lightlike vector $b \in V_0$ such that $\bp{a,b} = 1$ and let $c = \overline{p} - b$.
	\end{description}
	
	\noindent Now, define $N_0$ as follows:
	
	\begin{equation}
		N_0 := c + \{p \in V_0 | \bp{a_i, p} > 0 \text{ for all i }  \}
	\end{equation}
	
	Note that $N_0$ is an open subset of the plane determined by $(\overline{p},V_0,0)$. For $i=1,...,k$, let $N_i$ be the spherical submanifold of $\eunn$ determined by $(\overline{p},V_i,a_i)$. Define
	
	\begin{equation}
		\rho_i : \begin{cases}
			N_0 & \rightarrow \R_+ \\
			p_0 & \mapsto \bp{a_i, p_0 - c} = 1 + \bp{a_i, p_0 - \overline{p}}
		\end{cases}
	\end{equation}
	
	For $i = 1,...,k$, let $W_i := \R a_i \obot V_i$ and $P : \eunn \rightarrow W_i$ be the orthogonal projection. Then the map
	
	\begin{equation} \label{eq:psiGenFormEunn}
		\psi : \begin{cases}
		N_0 \times _{\rho_1} N_1 \times \cdots  \times _{\rho_k} N_k & \rightarrow \eunn \\
		(p_0,...,p_k) & \mapsto p_0 + \sum\limits_{i=1}^k \rho_i(p_0)(p_i - \overline{p})
		\end{cases}
	\end{equation}
	
	\noindent is an isometry onto the following set\footnote{Note that $\sgn 0 = 0$, otherwise for $a \neq 0$, $\sgn a$ is the sign of $a$.}:
	
\begin{equation}
	\Ima(\psi) := \begin{cases}
		c + \{ p \in \eunn \: | \: \sgn (P_i(p))^2 = \epsilon_i, \text{ for each } i = 1,...,k \} & \text{non-null case} \\
		c + \{ p \in \eunn \: | \: \bp{a,p} > 0 \} & \text{null case}
	\end{cases}
\end{equation}

%

$\Ima(\psi)$ is dense in $\eunn$ only for a non-null warped decomposition when each $W_i$ for $i = 1,...,k$ is Euclidean or anti-isometric to a Euclidean space.
\end{theorem}

\begin{remark} \label{rem:WPEunnCon}
	Note that $\rho_i(\bar{p}) = 1$ for $i=1,...,k$. Also for each $p_i \in N_i$ we have $\psi(\bar{p},\dotsc,p_i,\dotsc,\bar{p}) = p_i$, hence $\psi(\bar{p},\dotsc,\bar{p}) = \bar{p}$.

	If the $N_i$ are required to be connected, then $\Ima(\psi)$ has to be modified slightly. For each $N_i$ that is disconnected (see the remark following \cref{thm:spherSub}), in addition to the restriction that $\sgn (P_i(p))^2 = \epsilon_i$ in the definition of $\Ima(\psi)$, add the restriction that $\scalprod{a_i}{P_i(p)} > 0$.
\end{remark}

\begin{proof}
	The idea of this proof is to assume \cref{eq:psiGenFormEunn} holds and then expand it by choosing an appropriate basis for $V_0$. In the expanded form we will be able to prove all the claims made in the theorem. We have the following two cases.

	\textbf{The non-null case:}  Let $W_0$ be the orthogonal complement of $\bigobot\limits_{i=1}^k \R a_i$ in $V_0$; which is well defined since $a_i^2 \neq 0$ for each $i$. Thus we have that 
	
	\begin{equation} \label{eq:VnotDecomp}
		V_0 = W_0 \obot \bigobot\limits_{i=1}^k \R a_i
	\end{equation}
	
	\noindent which implies:
	
	\begin{align}
		\eunn & = \bigobot\limits_{i=0}^k V_i \nonumber \\
		& = W_0 \obot \bigobot\limits_{i=1}^k \R a_i \obot \bigobot\limits_{i=1}^k V_i \nonumber \\
		& = W_0 \obot \bigobot\limits_{i=1}^k (\R a_i \obot V_i) \nonumber \\
		& = W_0 \obot \bigobot\limits_{i=1}^k W_i
	\end{align}
	
	Now let $P_i : \eunn \rightarrow W_i $ denote the orthogonal projection for $i = 0,...,k$. Then from \cref{eq:VnotDecomp}, we get the following orthogonal decomposition of $V_0$ which will be used extensively:
	
	\begin{equation} \label{eq:nonNullVnotDecomp}
		p = P_0p + \sum\limits_{i=1}^k \frac{1}{\kappa_i} \bp{a_i,p}a_i \quad \text{ for all } p \in V_0
	\end{equation}
	
	Now we use the above decomposition of $p \in V_0$ to write $\psi(p_0,...,p_k)$ adapted to the following affine decomposition of $\eunn$
	
	\begin{equation}
		\eunn = c + \bigobot\limits_{i=0}^k W_i 
	\end{equation}
	
	We get the following for $(p_0,...,p_k) \in N_0 \times \cdots \times N_k$
	
	\begin{equation} \label{eq:nonNullImDecomp}
		\psi(p_0,...,p_k) = c + P_0(p_0 - c) + \sum\limits_{i=1}^k \bp{a_i, p_0-c} (p_i-c_i)
	\end{equation}
	
%
	
	Now we prove that $\psi$ is injective: Let $(p_0,...,p_k), (q_0,...,q_k) \in N_0 \times \cdots \times N_k$ and suppose that $\psi(p_0,...,p_k) = \psi(q_0,...,q_k)$. From \cref{eq:nonNullImDecomp}, we deduce the following:
	
	\begin{align}
		P_0(p_0 - c) & = P_0(q_0 - c) \\
		\bp{a_i, p_0-c} (p_i-c_i) & = \bp{a_i, q_0-c} (q_i-c_i)	
	\end{align}
	
	Since for each $i = 1,...,k$, $(p_i-c_i)^2 = (q_i-c_i)^2 = \frac{1}{\kappa_i}$ and $\bp{a_i, p_0-c}, \bp{a_i, q_0-c} \in \R^+$, we deduce that $p_i = q_i$. Then \cref{eq:nonNullVnotDecomp} shows $p_0 = q_0$.
	
	Now for surjectivity: From \cref{eq:nonNullImDecomp} it's clear that $\psi(N_0 \times \cdots \times N_k) \subseteq \Ima(\psi)$. Given $p\in \Ima(\psi)$, using \cref{eq:nonNullImDecomp} in conjunction with \cref{eq:nonNullVnotDecomp} we can readily calculate the inverse $q = \psi^{-1}(p)$ given in components as follows:
	
	\begin{align}
		q_0 & = c + P_0(p-c) + \sum\limits_{i=1}^k \frac{\epsilon_i}{\sqrt{|\kappa_i|}} \norm{P_i(p-c)} a_i \\
		q_i & = c_i + \frac{1}{\sqrt{|\kappa_i|}} \frac{P_i(p-c)}{\norm{P_i(p-c)}} \quad i = 1,...,k
	\end{align}
	
	Now we show that $\psi$ is an isometry. Note first that for $p = (p_0,...,p_k) \in N_0 \times \cdots \times N_k$ and $v = (v_0,...,v_k) \in T_p (N_0 \times \cdots \times N_k)$, \cref{eq:nonNullImDecomp} implies that
	
	\begin{equation}  \label{eq:nonNullWPPush}
		\psi_*v = P_0v_0 + \sum\limits_{i=1}^k \bp{a_i, v_0}(p_i-c_i) + \sum\limits_{i=1}^k \bp{a_i, p_0-c}v_i
	\end{equation}
	
	Hence also using the fact that:
	
	\begin{equation}
		\bp{p_i-c_i, v_i} = 0 \text{ for } i = 1,...,k
	\end{equation}
	
	\noindent we get:
	
	\begin{align}
		(\psi_*v)^2 & = (P_0v_0)^2 + \sum\limits_{i=1}^k  (\bp{a_i, v_0}(p_i-c_i))^2 + \sum\limits_{i=1}^k (\bp{a_i, p_0-c}v_i)^2 \\
		& = (P_0v_0)^2 + \sum\limits_{i=1}^k \frac{\bp{a_i, v_0}^2}{\kappa_i} + \sum\limits_{i=1}^k \rho_i(p_0)^2 v_i^2 \\
		& = (P_0v_0 + \sum\limits_{i=1}^k \frac{\bp{a_i, v_0}}{\kappa_i} a_i)^2  + \sum\limits_{i=1}^k \rho_i(p_0)^2 v_i^2 \\
		& = v_0^2 + \sum\limits_{i=1}^k \rho_i(p_0)^2v_i^2
	\end{align}
	
	\noindent where the last two lines follow from the fact that $v_0 \in V_0$ and \cref{eq:nonNullVnotDecomp}.
	
	\textbf{The null case:} We have the following decomposition of $V_0$:
	
	\begin{equation} \label{eq:nullVnotDecomp}
		V_0 = W_0 \obot \spa{a,b}
	\end{equation}
	
	\noindent where $W_0$ is the orthogonal complement of $\spa{a,b}$ relative to $V_0$. Let $P_i$ denote the orthogonal projection onto $W_0$ for $i = 0$ and onto $V_1$ for $i = 1$. Then for $p \in \eunn$:
	
	\begin{equation} \label{eq:nullDomDecomp}
		p = P_0p + \bp{b,p} a + \bp{a,p} b + P_1p
	\end{equation}
	
	\noindent and
	
	\begin{equation} \label{eq:nullDomDecompLen}
		p^2 = (P_0p)^2 + 2\bp{b,p}\bp{a,p} + (P_1p)^2
	\end{equation}
	
	Let $c = \overline{p} - b$, $\tilde{p}_0 = p_0 - c$ and $\tilde{p}_1 = p_1 - \overline{p}$, then for $(p_0,p_1) \in N_0 \times N_1$
	
	\begin{align}
		\psi(p_0,p_1) & = c + P_0(\tilde{p}_0) + (\bp{b, \tilde{p}_0} - \frac{1}{2}\bp{a, \tilde{p}_0}(P_1(\tilde{p}_1))^2)a + \bp{a, \tilde{p}_0}b \nonumber \\
		& \qquad + \bp{a, \tilde{p}_0} P_1(\tilde{p}_1) \label{eq:nullImDecomp}
	\end{align}
	
%
	
	\noindent where the last two lines follow from \cref{eq:nullDomDecompLen}.
	
	Injectivity of $\psi$ follows readily from  \cref{eq:nullImDecomp}.
	
	Now for surjectivity: From \cref{eq:nullImDecomp} it's clear that $\psi(N_0 \times N_1) \subseteq \Ima(\psi)$. Given $p \in \Ima(\psi)$, let $\tilde{p} = p -c$, then using \cref{eq:nonNullImDecomp} in conjunction with \cref{eq:nonNullVnotDecomp} we can readily calculate the inverse $q = \psi^{-1}(p)$ given in components as follows:
	
	\begin{align}
		q_0 & = c + P_0(\tilde{p}) + (\bp{b, \tilde{p}} +\frac{1}{2\bp{a, \tilde{p}}}(P_1(\tilde{p}))^2)a + \bp{a, \tilde{p}}b \label{eq:nullDecompInv} \\
		q_1 & = \overline{p} + \frac{1}{\bp{a, \tilde{p}}}P_1(\tilde{p}) - \frac{1}{2\bp{a, \tilde{p}}^2}(P_1(\tilde{p}))^2 a
	\end{align}
	
	Now we show that $\psi$ is an isometry. Note first that for $p = (p_0,p_1) \in N_0 \times N_1$ and $v = (v_0,v_1) \in T_p (N_0 \times N_1)$, \cref{eq:nullImDecomp} implies that
	
	\begin{align}
		\psi_*v & = P_0v_0 + (\bp{b, v_0} - \frac{1}{2}\bp{a, v_0}(P_1(\tilde{p}_1))^2 - \bp{a, \tilde{p}_0}\bp{P_1\tilde{p}_1, P_1v_1})a + \bp{a,v_0}b \nonumber \\
		& \qquad + \bp{a, v_0} P_1(\tilde{p}_1) + \bp{a, \tilde{p}_0} P_1(v_1) \label{eq:nullWPPush}
	\end{align}
	
	Hence we get that:
	
	\begin{align}
		(\psi_*v)^2 & = (P_0v_0)^2 + 2 \bp{b, v_0} \bp{a, v_0} + \bp{a, \tilde{p}_0}^2  (P_1v_1)^2 \\
		& = (\bp{b, v_0}a + \bp{a, v_0}b + P_0v_0)^2 + \rho(p_0)^2  (P_1v_1)^2 \\
		& = v_0^2 + \rho(p_0)^2v_1^2
	\end{align}
	
	\noindent where the last two lines follow from the fact that $v_0 \in V_0$, \cref{eq:nullDomDecompLen} and since $P_1 : T_{p_1} N_1 \rightarrow V_1$ is an isometry for each $p_1 \in N_1$.
\end{proof}

\begin{definition}
	We call $\psi$ the warped product decomposition of $\eunn$ determined by $(\overline{p}; N_1,...,N_k)$ or by $(\overline{p}; \bigobot\limits_{i=0}^k V_i; a_1,...,a_k)$ as in the hypothesis of the above theorem.
\end{definition}

Note that in the context of the above definition, the warped product decomposition is proper if each $a_i \neq 0$. For actual calculations we wish to work with canonical forms. The following definition will be particularly convenient.

\begin{definition}[Canonical form for Warped products of $\eunn$]
	We say that a proper warped product decomposition of $\eunn$ determined by $(\bar{p}; \bigobot\limits_{i=0}^k V_i; a_1,...,a_k)$ is in canonical form if: $\bar{p} \in V_0$ and $\bp{\bar{p}, a_i} = 1$.
\end{definition}

We note here that any proper warped product decomposition $\psi$ of $\eunn$ can be brought into canonical form by the translation $\psi \rightarrow \psi - c$. This follows from the above theorem by observing that $\bp{\bar{p} - c, a_i} = 1$ for each $i > 0$. The following corollary gives the standard warped product decompositions of $\eunn$ in canonical form.

\begin{corollary}[Canonical form for Warped products of $\eunn$] \label{cor:WPconForm}
	Let $\psi$ be a proper warped product decomposition of $\eunn$ determined by $(\bar{p}; \bigobot\limits_{i=0}^k V_i; a_1,...,a_k)$ which is in canonical form.

	Then the conclusions of \cref{thm:WPDecomps} simplify as follows:
	
	\begin{equation}
		N_0 = \{p \in V_0 | \bp{a_i, p} > 0 \text{ for all i }  \}
	\end{equation}
	
	\begin{equation}
		\rho_i = \bp{a_i, p_0}
	\end{equation}
	
	\begin{equation}
		\Ima(\psi) = \begin{cases}
			\{ p \in \eunn \: | \: \sgn (P_i(p))^2 = \epsilon_i, \text{ for each } i = 1,...,k \} & \text{non-null case} \\
			\{ p \in \eunn \: | \: \bp{a,p} > 0 \} & \text{null case}
		\end{cases}
	\end{equation}

	For $(p_0,...,p_k) \in N_0 \times \cdots \times N_k$, $\psi$ has the following form:
	
	\begin{equation}
		\psi(p_0,...,p_k) = \begin{cases}
					P_0p_0 + \sum\limits_{i=1}^k \bp{a_i, p_0} (p_i-c_i) & \text{non-null case} \\
					P_0p_0 + (\bp{b, p_0} - \frac{1}{2}\bp{a, p_0}(P_1(p_1))^2)a + \bp{a,p_0}b + \bp{a, p_0} P_1p_1 & \text{null case}
				\end{cases}
	\end{equation}

	Furthermore, the following equation holds:
	
	\begin{equation} \label{eq:wpdecpSp}
		\psi(p_0,...,p_k)^2 = p_0^2
	\end{equation}
\end{corollary}
\begin{proof}

First note that for the non-null case:

\begin{align}
	\bp{a_i,c} & = \bp{a_i,\bar{p} - \frac{a_i}{\kappa_i}} \\
	& = 1 - \frac{\bp{a_i, a_i}}{\kappa_i} \\
	& = 0
\end{align}

Similarly for the null-case:

\begin{align}
	\bp{c, a} & = \bp{\bar{p} - b, a} \\
	& = 0
\end{align}

Thus we see that

\begin{align}
	N_0 & = c + \{p \in  V_0 | \bp{a_i, p} > 0 \text{ for all i }  \} \\
	& = \{p \in  V_0 | \bp{a_i, p} > 0 \text{ for all i }  \}
\end{align}

The formula for $\Ima(\psi)$ follows similarly. Clearly $\rho_i(p_0) =  \bp{a_i, p_0 - c} = \bp{a_i, p_0}$. Now we break into cases.

\textbf{The non-null case:}

Note that $c \in W_0$, so $P_0c = c$, hence

\begin{align}
	\psi(p_0,...,p_k) & = c + P_0(p_0 - c) + \sum\limits_{i=1}^k \bp{a_i, p_0-c} (p_i-c_i) \\
	& = c + P_0(p_0 - c) + \sum\limits_{i=1}^k \bp{a_i, p_0} (p_i-c_i) -\sum\limits_{i=1}^k \bp{a_i,c} (p_i-c_i) \\
	& = P_0p_0 + \sum\limits_{i=1}^k \bp{a_i, p_0} (p_i-c_i)
\end{align}

It follows from the above equation that $\psi(p_0,...,p_k)^2 = p_0^2$.

\textbf{The null case:}
	
By \cref{eq:nullDomDecomp}, $c$ can be written as follows:

\begin{equation}
	c = P_0c + \bp{b,c} a
\end{equation}

Thus \cref{eq:nullImDecomp} reduces to

\begin{align} 
	\psi(p_0,p_1) & = c + P_0(p_0) - P_0c - \bp{b,c}a + (\bp{b, p_0} - \frac{1}{2}\bp{a, p_0}(P_1(\tilde{p}_1))^2)a + \bp{a,p_0}b \nonumber \\
	& \qquad + \bp{a, p_0} P_1(\tilde{p}_1) \\
	& = P_0(p_0) + (\bp{b, p_0} - \frac{1}{2}\bp{a, p_0}(P_1(p_1))^2)a + \bp{a,p_0}b \nonumber \\
			& \qquad + \bp{a, p_0} P_1(p_1) \\
\end{align}

\noindent In the last equation we used the fact that $P_1 \tilde{p}_1 = P_1 p_1$ since $\bar{p} \in V_0$.

Finally, it follows from the above equation that $\psi(p_0,p_1)^2 = p_0^2$.
\end{proof}

\section{Isometry groups of Spherical submanifolds of pseudo-Euclidean space*} \label{sec:IsoSphEunn}

Warped products of spaces of constant curvature are closely related to certain integrable subgroups of the isometry group due to the following fact \cite{Zeghib2011}:

\begin{proposition}[Lifting isometries from Killing distributions] \label{prop:liftIsomKil}
	Let $M = B \times_{\rho} F$ be a warped product and suppose $\tilde{f} : F \rightarrow F$ is an isometry of F. Then the lift $f$ defined by
	
	\begin{equation}
		f(x,y) := (x, \tilde{f}(y)), \quad (x,y) \in B \times F
	\end{equation}
	
	\noindent is an isometry of M.
\end{proposition}

Theorem~5.1 in \cite{Zeghib2011} shows conversely that given a certain integrable group action on a pseudo-Riemannian manifold $M$, one can obtain a warped product whose spherical foliation is invariant under the action of the group. Hence in spaces of constant curvature one can show that the above property characterizes warped products. In view of this, in this section we state the isometry groups which preserve the spherical submanifolds of pseudo-Euclidean space.

The isometry groups of $H_\nu^n$ and $S_\nu^n$ are well documented, see for example \cite[section~9.2]{barrett1983semi}. In this section we will describe the isometry group of $\punn$. This is given in \cite[lemma~6]{Nolker1996} for the case when $\nu = 0$; that proof should generalize easily. Although, we will give a different proof (motivated by Nolker's results) using our knowledge of warped product decompositions and \cref{prop:liftIsomKil}.

We denote the homogeneous isometry group (i.e. orthogonal group) of $\E^{n+2}_{\nu+1}$ by $O_{\nu + 1}(n+2)$ (see \cite{barrett1983semi}). Then we have the following:

\begin{proposition}
	Let $-a$ be the mean curvature vector of $\punn$. The isometry group of $\punn$ is:
	
	\begin{equation} 
		I(\punn) = \{T \in O_{\nu + 1} (n+2) \; | \; T a = a \}
	\end{equation}
	
	Furthermore suppose we fix an embedding of $\eunn$ by fixing a subspace $V \simeq \eunn$, then for $p \in V$ and $\tilde{p} \in V^\perp$ we have the following Lie group isomorphism:
	
	\begin{equation}
		\phi : \begin{cases}
		O(V) \ltimes V & \rightarrow I(\punn) \\
		(B, v) & \mapsto  \phi(B,v)
		\end{cases}
	\end{equation}
	
	where
	
	\begin{equation}
		\phi(B,v)(p + \tilde{p}) = \tilde{p} + B p + \bp{a,\tilde{p}} v - (\bp{B p, v} + \frac{1}{2}\bp{a, \tilde{p}}v^2))a
	\end{equation}
\end{proposition}
\begin{proof}
	Consider the warped product decomposition:
	
	\begin{equation}
	\psi(p_0, p) = \bp{a,p_0}b + \bp{a, p_0} p + (\bp{b, p_0} - \frac{1}{2}\bp{a, p_0}p^2)a
	\end{equation}
	
	 for $p_0 \in N_0$ and $p \in V$. Note that
	
	\begin{equation}
	\psi(b, p) = b + p  - \frac{1}{2}p^2 a 
	\end{equation}
	
	is a map onto $\punn$. As in \cref{eq:nullDecompInv}, one can deduce that the inverse of $\psi$ is
	
	\begin{align}
		q_0 & = (\bp{b, p} +\frac{1}{2\bp{a, p}}(P p)^2)a + \bp{a, p}b \\
		q_1 & = \frac{1}{\bp{a, p}}P p
	\end{align}
	
	Let $B \in O(V)$, $v \in V$ and define $Tp = Bp + v$ for $p \in V$. Now define $\hat{T}$ by:
	
	\begin{equation}
		\hat{T} : \begin{cases}
		\E^{n+2}_{\nu+1} & \rightarrow \E^{n+2}_{\nu+1} \\
		p & \mapsto  \psi (p_0 , T p_1)
		\end{cases}
	\end{equation}
	
	Since $\psi$ is a warped product decomposition, it follows by \cref{prop:liftIsomKil} that $\hat{T}$ induces an isometry of some open subset of $\E^{n+2}_{\nu+1}$ onto itself. We will now calculate $\hat{T}$ explicitly.
	
	For arbitrary $x \in \E^{n+2}_{\nu + 1}$ write $x = p + \tilde{p}$ where $p \in V$ and $\tilde{p} \in V^\perp$.
	
	\begin{align}
		(T q_1)^2 & = \norm{\frac{1}{\bp{a,x}} B p + v} \\
		& = (\frac{1}{\bp{a, x}^2} (P x)^2 + \frac{2}{\bp{a, x}} \bp{B p, v} + v^2)
	\end{align}
	
	\begin{align}
	\psi(q_0, T q_1) & = \bp{a,q_0}b + \bp{a, q_0} T q_1 + (\bp{b, q_0} - \frac{1}{2}\bp{a, q_0}(T q_1)^2)a \\
	& = \bp{a,x}b + \bp{a,x} T q_1 + (\bp{b, x} +\frac{1}{2\bp{a, x}}(P x)^2 - \frac{1}{2}\bp{a, x}(T q_1)^2)a \\
	& = \bp{a,x}b + \bp{a,x} T q_1 + (\bp{b, x} - (\bp{B p, v} + \frac{1}{2}\bp{a, x}v^2))a \\
	& = \bp{b, x} a + \bp{a,x}b + B p + \bp{a,x} v - (\bp{B p, v} + \frac{1}{2}\bp{a, x}v^2))a \\
	& = \tilde{p} + B p + \bp{a,x} v - (\bp{B p, v} + \frac{1}{2}\bp{a, x}v^2))a
	\end{align}
	
	Hence if $p := P x$ and $\tilde{p} := (I-P)x$ then
	
	\begin{equation}
		\hat{T}x = \tilde{p} + B p + \bp{a,x} v - (\bp{B p, v} + \frac{1}{2}\bp{a, x}v^2))a
	\end{equation}
	
	Thus since $\hat{T}$ is a linear isometry of $\E^{n+2}_{\nu+1}$ it follows that $\hat{T} \in O(\E^{n+2}_{\nu+1})$. Also $\hat{T}$ clearly fixes $a$ so $\hat{T} \in I(\punn)$.
	
	Let the map $\phi$ be as in the hypothesis. Note that $\phi(B,v) = \hat{T}$. $\phi$ is a Lie group homomorphism, since
	
	\begin{align}
		\hat{(T S)}x & = \psi (x_0 , T S x_1) \\
		& = \psi ((\hat{S} x)_0 , T (( \hat{S} x)_1) \\
		& = \hat{T} \hat{S} x
	\end{align}
	
	By definition of $\hat{T}$ it follows that $\phi$ is injective.
		
	
	To show that $\phi$ is surjective, fix $T \in I(\punn)$. Consider the decomposition:
	
	\begin{equation}
		p = \bp{a,p} b + \bp{b,p} a + P p, \quad p \in \eunn
	\end{equation}
	
	Using the fact that $(T p)^2 = p^2$ with the above decomposition we obtain the following equations:
	
	\begin{align}
		p^2 & = (T p)^2 = 2 \bp{a, Tp} \bp{b, Tp} + (P T p)^2  \\
		p \in V \Rightarrow p^2 & = (P T p)^2 \\
		p = b \Rightarrow  0 = b^2 & = 2 \bp{b, T b} + (P T b)^2 \\
		p = \psi(b, \tilde{p}) \Rightarrow  0 = p^2 & = 2 \bp{b, T \tilde{p}} + \bp{P T b, P T \tilde{p}}
	\end{align}
	
	The second equation implies that $P T \in O(V)$. We claim that $\phi(PT, PTb) = T$. This can be seen by decomposing the action of $T$ with respect to the above decomposition and then using the last three equations and the fact that $T \in I(\punn)$.
	
	Hence $\phi$ is a Lie group isomorphism.
\end{proof}

We also note that if $\psi : \eunn \rightarrow \punn$ is the standard embedding from \cref{eq:psiPunn}, then $\psi$ is equivariant, i.e. in the notation of the proof $\psi \circ T (p) = \hat{T} \circ \psi (p)$.

\section{Warped Product decompositions of Spherical submanifolds of Pseudo-Euclidean space} \label{sec:WPdecompSphEunn}

\subsection{Spherical submanifolds of \texorpdfstring{$\eunn(\kappa)$}{hyperquadrics of pseudo-Euclidean space}}

In this section we will classify the spherical submanifolds of $\eunn(\kappa)$. In particular we will show that they all have the form $\eunn(\kappa) \cap (\bar{p} + W)$ for some $\bar{p} \in \eunn$ and some subspace $W$. Although not all spherical submanifolds will have this form since we are only considering the case of pseudo-Riemannian manifolds. We will see that all spherical submanifolds of $\eunn(\kappa)$ arise as restrictions of spherical submanifolds of $\eunn$.

The following lemma concerns a submanifold $N$ of $\eunn(\kappa)$. We denote by $H'$ the mean curvature normal of $N$ in $\eunn(\kappa)$ and $H$ the mean curvature normal of $N$ in $\eunn$. Similar definitions hold for the second fundamental forms $h'$ and $h$. As usual $r$ denotes the dilatational vector field.

\begin{lemma} \label{lem:sphEunnkap}
	If $N$ is a submanifold of $\eunn(\kappa)$ then the following equations hold:
	
	\begin{equation}
		h(X,Y) = h'(X,Y) - \bp{X,Y} \frac{r}{r^2}
	\end{equation}
	
	\begin{equation} \label{eq:meanCurvRel}
		H = H' - \frac{r}{r^2}
	\end{equation}
	
	In particular, $N$ is an umbilical submanifold of $\eunn(\kappa)$ iff it is an umbilical submanifold of $\eunn$. In fact, $N$ is a spherical submanifold of $\eunn(\kappa)$ iff it is a spherical submanifold of $\eunn$.
\end{lemma}
\begin{proof}
	These formulas follow from lemma~3.5 and corollary~3.1 in \cite{chen2011pseudo}.
\end{proof}

Now we consider the problem of finding the sphere in $\eunn(\kappa)$ passing through a point $\overline{p}$ with tangent space $V$ and mean curvature normal $z$ at $\overline{p}$. We make this precise as follows.

Let $\overline{p} \in \eunn(\kappa)$ be arbitrary, $V \subset T_{\overline{p}}\eunn(\kappa)$ a non-degenerate subspace with $m := \dim V \geq 1$, $\mu := \ind V$ and $z \in V^{\perp} \cap T_{\overline{p}}\eunn(\kappa)$.

Now let $a := \kappa \overline{p}  - z$. Then assuming this data defines a submanifold of $\eunn(\kappa)$, we use \cref{eq:meanCurvRel} to obtain the mean curvature normal in $\eunn$ at $\bar{p}$, which is given as follows:

\begin{equation}
	z - \kappa\overline{p} = - a
\end{equation}

Then this determines a sphere in $\eunn$ with initial data $(\overline{p}, V, a)$ by \cref{thm:spherSub}. Note that $a \neq 0$. In the following theorem we will show that this sphere in $\eunn$ is in fact the sphere in $\eunn(\kappa)$ determined by $(\overline{p}, V, a)$. First let $W := \R a \obot V$ and $\tilde{\kappa} := a^2$.

\begin{theorem}[Spherical submanifolds of $\eunn(\kappa)$] \label{thm:spherKapSub}
	There is exactly one $m$-dimensional connected and geodesically complete spherical submanifold $\tilde{N}$ of $\eunn(\kappa)$ with $\bar{p} \in \tilde{N}$, $T_{\bar{p}} \tilde{N} = V$ and having mean curvature vector at $\bar{p}$, z. $\tilde{N}$ is an open submanifold of N; $N = \eunn(\kappa) \cap (\bar{p} + W)$ is the spherical submanifold determined by $(\overline{p},V,a)$ in $\eunn(\kappa)$ and $\eunn$. In fact, $N$ can be given explicitly as follows (where $\simeq$ means isometric to):
	
	\begin{enumerate}[(a),style=multiline]
		\item \label{it:sSphHyp} $a$ is timelike, then $\mu \leq \nu - 1$ and $N \simeq H^m_{\mu}(\tilde{\kappa})$
		\item \label{it:sSphSph} $a$ is spacelike, then $N \simeq S^m_{\mu}(\tilde{\kappa})$  \\
		For cases (b) and (c), let $c = \overline{p} -\frac{a}{\tilde{\kappa}}$ be the center of N, then N is given as follows:
		\begin{equation}
			N = c + \{ p \in W \: | \: p^2 = \frac{1}{\tilde{\kappa}} \}
		\end{equation}
		\item \label{it:sSphPar} $a$ is lightlike, then $\mu \leq \nu - 1$ and $N \simeq \E^m_{\mu}$
		\begin{equation}
			N = \overline{p} + \{ p - \frac{1}{2} p^2 a  \: | \: p \in V \}
		\end{equation}
	\end{enumerate}
\end{theorem}
\begin{remark}
	The relationship between $\tilde{N}$ and $N$ follows from \cref{rem:tildN}, since the above spheres are the spheres in $\eunn$ determined by $(\overline{p}, V, a)$ in \cref{thm:spherSub}.
\end{remark}
\begin{remark}
	$N$ is a geodesic submanifold of $\eunn(\kappa)$ iff $z = 0$ iff $W$ intersects the origin.
\end{remark}
\begin{proof}
	First we note that it suffices to show that there exists a single connected and geodesically complete sphere satisfying the initial conditions. By \cref{lem:unqSph}, it must be unique.

	The three definitions of $N$ given above follow directly from \cref{thm:spherSub} with initial data $(\overline{p}, V, a)$. Hence the relevant intrinsic properties of $N$ follow from \cref{thm:spherSub}. For the remainder of the proof we will assume $N$ is given by those definitions, and we will prove the following.
	
	\begin{claim}
		$N = \eunn(\kappa) \cap (\bar{p} + W)$
	\end{claim}
	\begin{proof}
		Note that the following equations are satisfied: $\bp{\bar{p},\bar{p}} = \dfrac{1}{\kappa}$, $\bp{a,\bar{p}} = 1$
	
		First we consider the case of \cref{it:sSphHyp,it:sSphSph}. We can always write $p = c + \tilde{p}$ where $\tilde{p} \in W$. Also note that the following holds:
		
		\begin{align}
			\bp{c,c} & = \bp{\bar{p},\bar{p}} - 2 \scalprod{\bar{p}}{\frac{a}{\tilde{\kappa}}} + \frac{1}{\tilde{\kappa}^2}\bp{a,a} \label{eq:lenC} \\
			& = \dfrac{1}{\kappa} - 2 \dfrac{1}{\tilde{\kappa}} +  \dfrac{1}{\tilde{\kappa}} \\
			& = \dfrac{1}{\kappa} - \dfrac{1}{\tilde{\kappa}}
		\end{align}
		
		Then since $\bp{c,a} = 0$, we have 
		
		\begin{align}
			\bp{p,p} & = c^2 - 2 \bp{c,\tilde{p}} + \tilde{p}^2 \\
			& = \dfrac{1}{\kappa} - \dfrac{1}{\tilde{\kappa}} + \tilde{p}^2
		\end{align}
		
		The above equation shows that $p \in \eunn(\kappa)$ iff $\tilde{p} \in W(\tilde{\kappa})$, which proves the result.
		
		Now for \cref{it:sSphPar}. We can always write $p = \bar{p} + v + w a$ where $v \in V$ and $w \in \R$. Hence
		
		\begin{align}
			\bp{p,p} & = \dfrac{1}{\kappa} + v^2 + 2 w
		\end{align}
		
		The above equation shows that $p \in \eunn(\kappa)$ iff $w = - \frac{1}{2} v^2$, which proves the result.	
	\end{proof}
	
	Thus we have shown that $N$ is a spherical submanifold of $\eunn$ contained in $\eunn(\kappa)$. It then follows from \cref{lem:sphEunnkap} that $N$ is a spherical submanifold of $\eunn(\kappa)$ with mean curvature normal $z$ at $\bar{p}$. Furthermore by \cref{prop:eunnUmb}~\ref{it:umbCc}, this sphere is of constant curvature $\kappa + z^2 = a^2 = \tilde{\kappa}$.
\end{proof}

Now we mention when we can restrict a sphere in $\eunn$ to one in $\eunn(\kappa)$. Suppose $(\bar{p}, V, -z)$ determines a sphere in $\eunn$ with $\bar{p} \in \eunn(\kappa)$ and $V \subset T_{\bar{p}}\eunn(\kappa)$. Then define $z'$ as

\begin{equation}
	z' := z + \kappa \overline{p} \in V^{\perp}
\end{equation}

We know that $\bar{p} \in V^{\perp}$ and $z \in V^{\perp}$ by hypothesis. In order for $\bp{z', \bar{p}} = 0$, we must additionally assume $\bp{z,\bar{p}} = - 1$. In this case, $(\bar{p}, V, -z)$ define initial data for a sphere in $\eunn(\kappa)$. It follows from the above theorem that this sphere is simultaneously the sphere in $\eunn$ and in $\eunn(\kappa)$ determined by $(\bar{p}, V, -z)$.

\subsection{Warped Product decompositions of Spherical submanifolds of Pseudo-Euclidean space}

Suppose $\psi : N_0 \times_{\rho_1} N_1 \times \cdots \times_{\rho_k} N_k \rightarrow \eunn(\kappa) $ is a warped product decomposition of $\eunn(\kappa)$ associated with initial data $(\bar{p};\bigobot\limits_{i=0}^k V_i; a_1,...,a_k)$ where each $a_i = \kappa \bar{p} - z_i$. By \cref{eq:wpSccMean}, the mean curvature vectors at $\bar{p}$ satisfy the following equation:

\begin{equation}
	\bp{z_i,z_j} = - \kappa \quad i \neq j
\end{equation}

By \cref{thm:spherKapSub}, $L_i(\bar{p})$ is a spherical submanifold of $\eunn$ determined by $(\bar{p},V_i, a_i)$. Note that $a_i \neq 0$. Furthermore the above equation implies that

\begin{equation}
	\bp{a_i,a_j} = 0 \quad i \neq j
\end{equation}

Also recall that by assumption, the $a_i$ are linearly independent. Thus the initial data $(\bar{p}; (\R \bar{p} \obot V_0) \obot V_1 \obot \dots  \obot V_k; a_1, \dotsc,  a_k)$ determines a proper warped product decomposition of the ambient space $\eunn$. Furthermore, we note that this warped product decomposition is in canonical form; the canonical form was specifically designed to have this property. We now consider the converse problem of restricting a warped product decomposition in $\eunn$ to $\eunn(\kappa)$. The following theorem shows that this is always possible when the warped product in $\eunn$ is proper and in canonical form:

\begin{theorem}[Restricting Warped products to $\eunn(\kappa)$] \label{thm:eunnKapRestWP}
	Let $\psi$ be a proper warped product decomposition of $\eunn$ associated with $(\bar{p}; \bigobot\limits_{i=0}^k V_i; a_1,...,a_k)$ which is in canonical form. Suppose $\kappa^{-1} := \bar{p}^2 \neq 0$ and let $N' := N_0(\kappa) \times_{\rho_1} N_1 \times \cdots \times_{\rho_k} N_k$. Note that $N_0(\kappa)$ is an open subset of the sphere in $\eunn(\kappa)$ determined by $(\bar{p},(\bar{p}^{\perp} \cap V_0),0)$. Then $\phi : N' \rightarrow \eunn(\kappa)$ defined by $\phi := \psi|_{N'}$ is a warped product decomposition of $\eunn(\kappa)$ determined by $(\bar{p}; (\bar{p}^{\perp} \cap V_0)\bigobot\limits_{i=1}^k V_i; a_1,...,a_k)$.
	
	Furthermore for any point $p \in \Ima(\psi)$ with $p^2 \neq 0$, the leaf of the foliation induced by $N_i$, $L_i(p)$, is simultaneously a sphere in $\eunn$ and $\eunn(\frac{1}{p^2})$. Also $\psi$ is in canonical form at every $p \in \Ima(\psi)$. 
\end{theorem}
\begin{proof}
	By \cref{eq:wpdecpSp} in \cref{cor:WPconForm} it follows that $\phi$ is a diffeomorphism onto $\phi(N') \subseteq \eunn(\kappa)$. Clearly the restriction of the metric on $N$ to $N'$ is still a warped product metric. Hence it follows that $\phi$ is a warped product decomposition of $\eunn(\kappa)$, i.e. an isometry from a warped product. Furthermore by \cref{thm:spherKapSub} it follows that for each $i > 0$, $N_i$ is also the sphere in $\eunn(\kappa)$ determined by $(\bar{p}, V_i, z_i)$.
	
	Now for the last point, fix $p \in \Ima(\psi)$ with $p^2 \neq 0$. Let $\tilde{r}$ be the dilatational vector field in $N_0$ and $r := \psi_*\tilde{r}$. Can show that $r$ is also the dilatational vector field in $\eunn$ (e.g. see \cref{eq:psiGenFormEunn}). Now if $\rho_i = \bp{\tilde{r}, a_i}$, then it can be shown using results in \cite{Meumertzheim1999} that the mean curvature vector $H_i$ is:
	
	\begin{equation}
		H_i = - \frac{a_i}{\rho_i}
	\end{equation}
	
	Hence $\bp{\tilde{r}, -H_i} = 1$. Thus at $p$, by making the identification $r = p$, we see that $T N_i$ is orthogonal to $p = \psi_* \tilde{p}$ and $\bp{\tilde{p}, -H_i} = 1$. It follows from the discussion following \cref{thm:spherKapSub} that $L_i(p)$ is also a sphere in $\eunn(\frac{1}{p^2})$.
\end{proof}
\begin{remark}[Connectedness]
	\Cref{rem:WPEunnCon} gives the appropriate modifications of $\Ima(\psi)$ when each $N_i$ for $i > 0$ are required to be connected. When $N_0(\kappa) \simeq H^m_0(\tilde{\kappa})$ or $N_0(\kappa) \simeq S^m_m(\tilde{\kappa})$, $N_0(\kappa)$ is disconnected \cite[Section~4.6]{barrett1983semi} and so we modify $N_0(\kappa)$ as follows: By \cref{thm:spherKapSub} it follows that $N_0(\kappa)$ is an open subset of the sphere in $\eunn$ determined by $(\bar{p},(\bar{p}^{\perp} \cap V_0),\kappa \bar{p})$. Thus to enforce connectedness, it follows by \cref{rem:tildN} that we must replace $N_0(\kappa)$ with 
	\begin{equation}
		N_0(\kappa) \cap \{p \in V_0 \: | \: \bp{\kappa \bar{p},p} > 0 \}
	\end{equation}

\begin{proof}
	$a = \kappa \bar{p}$, $\tilde{\kappa} = a^2 = \kappa$
	\begin{align}
		c & = \bar{p} - \frac{a}{\tilde{\kappa}} \\
		& = 0
	\end{align}
\end{proof}

Now we show the effect of this on $\phi(N')$ when $\nu = 1$.

\begin{parts} 
	\item $a_i$ is time-like for some $i$ \\
	$N_0(\kappa)$ is automatically connected since $N_0(\kappa) \subset \{p \in V_0 \: | \: \bp{a_i,p} > 0 \text{ for each i} \}$, then $N_0(\kappa) \subset \{p \in V_0 \: | \: \bp{\kappa \bar{p},p} > 0 \}$ since $\bp{a_i,\kappa \bar{p}} = \kappa <0$ (see \cite[P.~143]{barrett1983semi} and Nolker's proof of the hyperbolic case).
	\item null case, $a := a_1$ is light-like \\
	$N_0(\kappa)$ is connected here as well. First observe that it follows from the equation for $\psi$ in \cref{cor:WPconForm} that
	\begin{equation}
		\bp{a, \psi(p_0,p_1)} = \bp{a, p_0} > 0
	\end{equation}
	
	Thus it follows that $N_0(\kappa)$ and $\phi(N')$ are in the time cone opposite to $a$ (see remarks preceding Nolker's proof of the hyperbolic case). Thus it follows that $N_0(\kappa) \subset \{p \in V_0 \: | \: \bp{\kappa \bar{p},p} > 0 \}$, so $N_0(\kappa)$ and hence $\phi(N')$ are connected.
	
	In this case $\phi(N')$ is the maximal connected component of $\eunn(\kappa)$ passing through $\bar{p}$. 
	\item $a_i$ is space-like for each $i$ \\
	First observe that it follows from the proof of \cref{cor:WPconForm} that $c = P_0c \in W_0$ and $p_i-c_i \in W_i$ for $i > 0$, hence
	
	\begin{align}
		\scalprod{c}{ \psi(p_0,\dotsc,p_k)} & = \scalprod{ c}{ P_0p_0 + \sum\limits_{i=1}^k \bp{a_i, p_0} (p_i-c_i)} \\
		& = \scalprod{ c}{ P_0p_0 } \\
		& = \bp{c, p_0}
	\end{align}
	
	Also since since $\bp{c,a_i} = 0$, we have that
	
	\begin{align}
		\bp{c,c} & = \bp{c,\bar{p}} \\
		& = \bp{\bar{p} - \sum\limits_{i=1}^k \frac{a_i}{\kappa_i}, \bar{p}} \\
		& = \frac{1}{\kappa} - \sum\limits_{i=1}^k \frac{1}{\kappa_i} \\
		& < 0 
	\end{align}
	
	In other words, $c$ is time-like. Also the above equation shows that $\bp{c,\kappa \bar{p}} > 0$, thus $c$ and $\kappa \bar{p}$ are in opposite time cones (see \cite[P.~143]{barrett1983semi}). Hence,
	
	\begin{equation}
		\{p \in V_0 \: | \: \bp{\kappa \bar{p},p} > 0 \} = \{p \in V_0 \: | \: \bp{\kappa c,p} > 0 \}
	\end{equation}
	
	Thus since $\scalprod{c}{ \psi(p_0,\dotsc,p_k)} = \bp{c, p_0}$, we see that $\phi(N')$ becomes 
	
	\begin{equation} 
		\phi(N') \cap \{ p \in \eunn \: | \: \bp{\kappa c,p} > 0 \}
	\end{equation}
	
%
\end{parts}

\end{remark}

In the following corollary we show how to obtain any warped product decomposition of $\eunn(\kappa)$ by restricting an appropriate warped product decomposition of $\eunn$. The ``appropriate'' warped product product decomposition of $\eunn$ to restrict follows from the discussion preceding the above theorem. Thus together with the above theorem, we have the following corollary:

\begin{corollary}[Warped product decompositions of $\eunn(\kappa)$]
	Suppose $(\bar{p};\bigobot\limits_{i=0}^k V_i; a_1,...,a_k)$ define initial data for a warped product decomposition of $\eunn(\kappa)$.

	Let $\phi$ be the warped product decomposition of $\eunn(\kappa)$ given in the above theorem by restricting the warped product decomposition of $\eunn$ with initial data $(\bar{p}; (\R \bar{p} \obot V_0) \bigobot\limits_{i=1}^k V_i; a_1, \dotsc,  a_k)$.
	
	Then $\phi$ is a warped product decomposition of $\eunn(\kappa)$ determined by $(\bar{p};\bigobot\limits_{i=0}^k V_i; a_1,...,a_k)$.
\end{corollary}

We now mention which warped product decompositions are possible in $\eunn(\kappa)$. We do this by finding out when it's possible to restrict a warped product on the ambient space. Given a warped product $ (V_0 \obot V_1 \obot \dots  \obot V_k; a_1, \dotsc,  a_k)$ passing through an arbitrary point in $\eunn$, in order to restrict it to $\eunn(\kappa)$, we need it to pass through a point $\bar{p} \in V_0$ with $\bar{p}^2 = \kappa$ satisfying $\bp{\bar{p}, a_i} = 1$. So for a fixed $\kappa \neq 0$, we enumerate the distinct warped products in $\eunn$, expand $\bar{p} \in V_0$ so that $\bp{\bar{p}, a_i} = 1$ and determine if it's possible for $\bar{p}^2 = \kappa$. By making use of \cref{thm:wpInEunn}, we have the following results:

\begin{theorem}[Warped products in Spherical submanifolds of $\E^n$ and $M^n$] \label{thm:wpInEunnKap}
	Suppose $N = N_0 \times_{\rho_1} N_1 \times \cdots \times_{\rho_k} N_k$ is a warped product decomposition of an open subset of a spherical submanifold of $\E^n$ or $M^n$. This warped product is necessarily proper. If at most one of the $N_i$ are intrinsically flat, then N is isometric to one of the following warped products:

	In $S^n$:
	\begin{align}
		S^m \times _{\rho_1} S^{n_1} \times \cdots \times _{\rho_s} S^{n_s}
	\end{align}
	
	In $dS^n$:
	\begin{align}
		dS^m \times _{\lambda_1} \E^{n_1} \times _{\rho_2} S^{n_2} \times \cdots \times _{\rho_s} S^{n_s} \\
		dS^m \times _{\tau_1} H^{n_1} \times _{\rho_2} S^{n_2} \times \cdots \times _{\rho_s} S^{n_s} \\
		S^m \times _{\rho_1} dS^{n_1} \times _{\rho_2} S^{n_2} \times \cdots \times _{\rho_s} S^{n_s} \\
		dS^m \times _{\rho_1} S^{n_1} \times _{\rho_2} S^{n_2} \times \cdots \times _{\rho_s} S^{n_s}
	\end{align}

	In $H^n$:
	\begin{align}
		H^m \times _{\lambda_1} \E^{n_1} \times _{\rho_2} S^{n_2} \times \cdots \times _{\rho_s} S^{n_s} \\
		H^m \times _{\tau_1} H^{n_1} \times _{\rho_2} S^{n_2} \times \cdots \times _{\rho_s} S^{n_s} \\
		H^m \times _{\rho_1} S^{n_1} \times _{\rho_2} S^{n_2} \times \cdots \times _{\rho_s} S^{n_s}
	\end{align}
	
	\noindent where $\nabla \rho_i$,$\nabla \tau_i$,$\nabla \lambda_i$ is a spacelike,timelike, lightlike vector field respectively.
\end{theorem}
\begin{proof}
	For the proof that the warped products are proper, see \cref{lem:EunnKapProdDecomp}.
\end{proof}

\section*{Acknowledgments}

I would like to thank Spiro Karigiannis for reading my thesis \cite{Rajaratnam2014}, which contains the contents of this article.